\documentclass[11pt]{article}
\usepackage{amssymb,amsmath,amsfonts,mathrsfs,amsthm,epsfig,latexsym,color}
\usepackage{enumitem,geometry}
\usepackage{fourier}
\usepackage[all]{xy}
\usepackage{xcolor}

\makeatletter
\newcommand{\subsectionruninhead}{\@startsection{subsection}{2}{0mm}
	{-\baselineskip}{-0mm}{\bf\large}}
\newcommand{\subsubsectionruninhead}{\@startsection{subsubsection}{3}{0mm}
	{-\baselineskip}{-0mm}{\bf\normalsize}}
\makeatother

\newtheorem*{theorem*}{Theorem}
\newtheorem*{proof*}{Proof}
\newtheorem*{proposition*}{Proposition}
\newtheorem*{notation*}{Notation}
\newtheorem*{corollary*}{Corollary}
\newtheorem*{claim*}{Claim}
\newtheorem*{remark*}{Remark}

\newtheorem{theorem}{Theorem}[section]
\newtheorem{proposition}{Proposition}[section]

\newtheorem{corollary}[proposition]{Corollary}
\newtheorem{lemma}[proposition]{Lemma}
\newtheorem{claim}[proposition]{Claim}

\theoremstyle{definition}
\newtheorem{definition}[proposition]{Definition}
\theoremstyle{remark}
\newtheorem{remark}[proposition]{Remark}

\numberwithin{equation}{section}

\def\NN{\mathbb{N}}
\def\RR{\mathbb{R}}

\def\TT{\mathbb{T}}
\def\ZZ{\mathbb{Z}}
\def\tildeL{\tilde{\mathcal{L}}}
\def\tildeF{\tilde{\mathcal{F}}}
\def\e{{\varepsilon}}

\begin{document}
	\title{Rigidity of stable Lyapunov exponents and integrability for Anosov maps}
	\author{Jinpeng An, ~ Shaobo Gan, ~ Ruihao Gu, ~ Yi Shi}
	\date{\today}
	\maketitle
	
	\begin{abstract}
	Let $f$ be a non-invertible irreducible Anosov map on $d$-torus. We show that if the stable bundle of $f$ is one-dimensional, then $f$ has the integrable unstable bundle, if and only if, every periodic point of $f$ admits the same Lyapunov exponent on the stable bundle with its linearization.  For  higher-dimensional stable bundle case, we get the same result on the assumption that $f$ is a $C^1$-perturbation of a linear Anosov map with real simple Lyapunov spectrum on the stable bundle. In both cases, this implies if $f$ is topologically conjugate to its linearization, then the conjugacy is smooth on the stable bundle.
	\end{abstract}
	
	\section{Introduction}\label{sec:introduction}
	
	 Let $M$ be a $d$-dimensional smooth closed Riemannian manifold. A diffeomorphism $f:M\to M$ is Anosov if there exists a continuous $Df$-invariant splitting $TM=E^s\oplus E^u$ such that $Df$ is uniformly contracting in $E^s$ and $Df$ is uniformly expanding in $E^u$. The classical Stable Manifold Theorem (e.g. \cite{pesinbook}) shows that both $E^s$ and $E^u$ are uniquely integrable. So there are $f$-invariant stable and unstable foliations tangent to $E^s$ and $E^u$ respectively.
	
	 The most well-known example of Anosov diffeomorphisms is a linear automorphism $A\in{\rm GL}_d(\ZZ)$ with all eigenvalues whose absolute values are not equal to $1$. The induced diffeomorphism $A:\TT^d\to\TT^d$ is Anosov.
	 All known Anosov diffeomorphisms are conjugate to affine automorphisms of infra-nilmanifolds. In particular, every Anosov diffeomorphism with ${\rm dim}E^s=1$ or ${\rm dim}E^u=1$ must be supported on $\TT^d$ \cite{Franks70}, and every Anosov diffeomorphism $f:\TT^d\to\TT^d$ is topologically conjugate to its linearization $f_*:\pi_1(\TT^d)\to\pi_1(\TT^d)$ acting on $\TT^d$\cite{Franks69,Manning74}.
	
	 In 1974, Ma\~n\'e and Pugh extended the concept of Anosov diffeomorphisms to non-invertible Anosov maps.
	
	\begin{definition}[\cite{manepugh}]\label{defmanepugh}
		A $C^1$ local diffeomorphism $f:M\to M$ is called  \textit{Anosov map}, if there exists a $Df$-invariant continuous subbundle $E^s\subset TM$ such that it is uniformly $Df$-contracting and its quotient bundle $TM/E^s$ is unformly $Df$-expanding.
	\end{definition}
	
	The set of Anosov maps on $M$ is $C^1$-open in the space $C^r(M)$ which consists of all $C^r$-maps of $M$ with $r\geq1$. All known Anosov maps are conjugate to affine endomorphisms of infra-nilmanifolds.
	
	Differing from Anosov diffeomorphisms,  there is a priori no $Df$-expanding subbundle $E^u\subset TM$ for a non-invertible Anosov map because the negative orbit for a point is not unique. For instance, Przytycki \cite{Przytycki} constructed a class of Anosov maps on torus which has infinitely many expanding directions on certain points. In fact, the set of expanding directions on a certain point in Przytycki's example contains a curve homeomorphic to interval in the $({\rm dim}M-{\rm dim}E^s)$-Grassman space \cite[Theorem 2.15]{Przytycki}.
	In the same paper \cite{Przytycki},
	Przytycki defined Anosov maps in the way of orbit space (see Definition \ref{defprz}) which allowed us to define the unstable bundle along every orbit. However, these unstable bundles are not integrable in general when project on the manifold $M$.

	We say an Anosov map $f$ has an \textit{integrable unstable bundle}, if there exists a continuous $Df$-invariant splitting $TM=E^s\oplus E^u$, such that $Df$ is uniformly contracting on $E^s$ (\textit{stable bundle}) and expanding on $E^u$ (\textit{unstable bundle}). Here $E^u$ is uniquely integrable, see \cite{Przytycki}.
	For example,
	$$
	A_0=\begin{bmatrix}
		3&1\\
		1&1
	\end{bmatrix} :\mathbb{T}^2\to\mathbb{T}^2,
	$$
	is an Anosov map on torus with  integrable unstable bundle.
	
	There are plenty of Anosov maps without integrable unstable bundles.
	Actually, Przytycki \cite[Theorem 2.18]{Przytycki} showed that any non-invertible Anosov map $f$ on any manifold $M$ with non-trivial stable bundle can be $C^1$-approximated by Anosov maps without integrable unstable bundles. Moreover, for every transitive Anosov map without integrable unstable bundle,
	it must have a residual set in the manifold in which every point has infinitely many expanding directions (\cite{tahzibi}).

	\vspace{0.3cm}
	
	In this paper, we give an equivalent characterization for a class of Anosov maps on $d$-torus $\TT^d$ which has integrable unstable bundle.
	
	Let $f:\mathbb{T}^d\to \mathbb{T}^d$ be an Anosov map on torus,
	then $f$ is homotopic to a linear toral map $A=f_*:\pi_1(\TT^d)\to\pi_1(\TT^d)$.
	Here $A$ is also an Anosov map \cite[Theorem 8.1.1]{aoki}  and is called the \textit{linearization} of $f$.
	A toral Anosov map $f$ is called \textit{irreducible}, if its linearization $A\in GL_d(\mathbb{R}) \cap M_d(\mathbb{Z})$ has irreducible characteristic polynomial over $\mathbb{Q}$.

	\begin{theorem}\label{main theorem 1}
		Let $f:\TT^d\to\TT^d$ be a $C^{1+\alpha}$ irreducible non-invertible Anosov map with one-dimensional stable bundle.
		Then $f$ has integrable unstable bundle, if and only if, every periodic point of $f$ admits the same Lyapunov exponent on the stable bundle.
	\end{theorem}
	
	\begin{remark}\label{C1 necessity}
		In both cases, the Lyapunov exponent of $f$ on the stable bundle is equal to its linearization $f_*$ on the stable bundle.
		Moreover, the necessity only need $C^1$ regularity of $f$: if a $C^1$ irreducible non-invertible Anosov map $f:\TT^d\to\TT^d$ with ${\rm dim}E^s=1$ has integrable unstable bundle, then every periodic point of $f$ admits the same stable Lyapunov exponent to its linearization $f_*$.
	\end{remark}

     In fact, an Anosov map on torus is conjugate to its linearization if and only if it admits an integrable unstable bundle \cite{specialconjugate}. A direct corollary is the following, which is an interesting example of rigidity in smooth dynamics, in the sense of ``weak equivalence''(topological conjugacy) implies ``strong equivalence''(smooth conjugacy).

    \begin{corollary}\label{cor-1}
    	Let $f:\TT^d\to\TT^d$ be a $C^{1+\alpha}$ irreducible non-invertible Anosov map with one-dimensional stable bundle. If $f$ is topologically conjugate to its linearization $f_*:\TT^d\to\TT^d$, then the conjugacy is $C^{1+\alpha}$-smooth along stable foliation.
    \end{corollary}

    In particular, we have the following corollary on two torus $\TT^2$.

    \begin{corollary}\label{cor-T2}
    	Let $f:\TT^2\to\TT^2$ be a non-invertible Anosov map, then the following are equivalent:
    	\begin{itemize}
    		\item $f$ has integrable unstable bundle;
    		\item $f$ is topologically conjugate to its linearization $f_*:\TT^2\to\TT^2$.
    	\end{itemize}
        Both of them imply the conjugacy between $f$ and $f_*$ is $C^{1+\alpha}$-smooth along the stable foliation.
    \end{corollary}

    \begin{remark}\label{Rk-1}
    	Recently, Micena \cite[Theorem 1.10]{micena} shows that for a $C^\infty$ non-invertible Anosov map $f:\TT^2\to\TT^2$ with integrable unstable bundle, if it admits periodic data on the stable and unstable bundle, then $f$ is $C^\infty$-conjugate to $f_*$ (also see \cite[Theorem C]{cantarino2021anosov}). Our result shows that we only need to assume $f$ admits periodic data on the unstable bundle, then it is $C^\infty$-conjugate to $f_*$.
    \end{remark}

    For higher-dimensional stable bundle case, we prove a local rigidity result for linear Anosov maps on $\TT^d$ with real simple spectrum.
    We say a hyperbolic matrix $A\in M_d(\ZZ) \cap GL_d(\RR)$ has \textit{real simple spectrum on stable bundle}, if all eigenvalues on the stable bundle are real and have mutually distinct moduli. Then $A$ admits a dominated splitting
    $$
    T\TT^d=L^s_1\oplus\cdots\oplus L^s_k\oplus L^u,
    $$
    with ${\rm dim}L^s_i=1$ for $i=1,\cdots,k$.

    If a map $f:\TT^d\to\TT^d$ is $C^1$-close to a hyperbolic $A\in M_d(\ZZ) \cap GL_d(\RR)$ with real simple spectrum on stable bundle, then $f$ is Anosov and has $k$-Lyapunov exponents on the stable bundle:
    $$
    \lambda^s_1(p,f)<\lambda^s_2(p,f)<\cdots\cdots<
    \lambda^s_k(p,f)<0,
    \qquad \forall p\in{\rm Per}(f).
    $$
    We say $f$ has \textit{spectral rigidity on  stable bundle} if for every periodic point $p\in{\rm Per}(f)$, it satisfies
    $$
    \lambda^s_i(p,f)=\log|\mu_i|, \qquad i=1,\cdots,k.
    $$
    Here $\mu_i$ is the eigenvalue of $A$ in the eigenspace  $L^s_i$.

    \begin{theorem}\label{main theorem 2}
    	Let $A\in M_d(\ZZ) \cap GL_d(\RR)$ be hyperbolic and irreducible with real simple spectrum on stable bundle. If $A:\TT^d\to\TT^d$ is non-invertible, then for every $f\in C^{1+\alpha}(\TT^d)$ which is $C^1$-close to $A$, $f$ has integrable unstable bundle, if and only if, it has spectral rigidity on  stable bundle.
    \end{theorem}

    Since $f$ having integrable unstable bundle is equivalent to $f$ being topologically conjugate to $A$, we have the following corollary.

    \begin{corollary}\label{cor-d}
    	Let $A\in M_d(\ZZ) \cap GL_d(\RR)$ be hyperbolic and irreducible with real simple spectrum on stable bundle. Assume $A:\TT^d\to\TT^d$ is non-invertible and $f\in C^{1+\alpha}(\TT^d)$ is $C^1$-close to $A$. If $f$ is topologically conjugate to $A$, then the conjugacy is $C^{1+\beta}$-smooth  along stable foliation, for some $0<\beta\leq \alpha$.
    \end{corollary}
    \begin{remark}
    	Here we lose the regularity of conjugacy because the weak stable bundle may only be $C^{\beta}$ continuous for some $0<\beta\leq\alpha$.
    \end{remark}

	We mention that the irreducible condition is necessary for our result. Indeed,
	$$
	A_1=\begin{bmatrix}
		2&1&0\\
		1&1&0\\
		0&0&2
	\end{bmatrix} :\mathbb{T}^3\to\mathbb{T}^3,
	$$
	is a non-invertible Anosov map with one-dimensional stable bundle and integrable unstable bundle, and it can be treated as a product system on $\mathbb{T}^2 \times S^1$.
	Note that one of its factor systems
	$$
	A_2=\begin{bmatrix}
		2&1\\
		1&1
	\end{bmatrix} :\mathbb{T}^2\to\mathbb{T}^2,
	$$
	is an Anosov diffeomorphism. So, we can make a smooth perturbation $f:\mathbb{T}^2\to\mathbb{T}^2$ of $A_2$ which is not smooth conjugate to $A_2$.
	Then the product map
	$$
	F=\begin{bmatrix}
		f&0\\
		0&2
	\end{bmatrix} :\mathbb{T}^3\to\mathbb{T}^3,
	$$
	is still an Anosov map with integrable unstable bundle, but it loses the rigidity of Lyapunov exponents on stable bundle.
	
	\vspace{0.3cm}
	
	We would like to give another view of Theorem \ref{main theorem 1} and Theorem \ref{main theorem 2}.  In \cite{GogolevGuysinsky08}, Gogolev and Guysinsky show that for an Anosov diffeomorphism $g:\TT^3\to\TT^3$ with partially hyperbolic splitting $T\TT^3=E^{ss}\oplus E^{ws}\oplus E^u$, if $g$ has spectral rigidity on the weak stable bundle $E^{ws}$, then $E^{ss}\oplus E^u$ is integrable. See \cite{Gogolevhighdimrigidity,GKS2011}  for higher-dimensional Anosov diffeomorphisms on $\TT^d$.  Conversely, Gan and Shi \cite{ganshirigidity} proved that if $E^{ss}\oplus E^u$ is integrable, then $g$ has spectral rigidity on $E^{ws}$. See \cite{gogolevshi} for higher-dimensional Anosov diffeomorphisms on $\TT^d$.
	
	For a non-invertible Anosov map $f:\TT^d\to\TT^d$ and every $x\in\TT^d$, we can see the preimage set of $x$:
	$$
	{\rm Perimage}(x)=\bigcup_{k\geq0}{\rm P}_k(x)
	\qquad \text{where} \qquad
	{\rm P}_k(x) = \{z\in\mathbb{T}^d|f^k(z)=f^k(x) \}
	$$
	as the strongest stable manifold of $x$. So the stable bundle of $f$ is corresponding to the weak stable bundle $E^{ws}$ of the Anosov diffeomorphism $g$. It is clear that $f$ has integrable unstable bundle implies for every $x\in\TT^d$, the unstable bundle of $x$ is independent of the choice of negative orbits of $x$. So we can see the unstable bundle is jointly integrable with the strongest stable bundle, which is corresponding to the case that $E^{ss}\oplus E^u$ is integrable for the Anosov diffeomorphism $g$. Thus we can expect the Anosov map has some rigidity on the stable bundle.
	
	\vskip3mm

	 The regularity of conjugacy for Anosov diffeomorphisms under the assumption of rigidity for Lyapunov exponents of periodic points has been extensively studied by many researchers e.g. \cite{delallave,Gogolevhighdimrigidity,GKS2011,Gogolev2017}. Recently, there are elegant works about the smooth conjugacy for conservative Anosov diffeomorphisms under the assumption of rigidity for Lyapunov exponents with respect to Lebesgue measures e.g. \cite{Yangjiagangradu,GogolevKalinin}. Our Corollary \ref{cor-1} and Corollary \ref{cor-d} show that we only need to assume spectral rigidity on the unstable bundle to get smooth conjugacy for non-invertible Anosov map with integrable unstable bundles, see Remark \ref{Rk-1} and \cite{micena}.
	
	 Finally, we would like to mention that Anosov maps on $\TT^2$ are a special class of partially hyperbolic maps on surfaces. A series of impressive results on SRB measures and statistical properties for partially hyperbolic maps on surfaces have been obtained, see \cite{tsujiiacta,liverani2,liverani1}. Meanwhile, the classification of partially hyperbolic endomorphisms on $\TT^2$ up to leaf conjugacy has  also been studied, see \cite{HeShiWang, Hammerlindl-Hall-21, Hammerlindl-Hall-22}. It will be interesting to classify all partially hyperbolic maps with integrable unstable bundle on surfaces.

	\vskip 0.5 \baselineskip
	
	\noindent {\bf Organization of this paper:}
	In section \ref{sec:preliminaries}, we recall some general properties of Anosov maps and give some useful properties on the assumptions of Theorem \ref{main theorem 1} and Theorem \ref{main theorem 2}.
    In section \ref{s-rigidity}, we prove the "necessary" parts of Theorem \ref{main theorem 1} and Theorem \ref{main theorem 2} with $C^1$ regularity, which state that
    the existence of integrable unstable bundle implies the spectral rigidity on stable bundle.
    In section \ref{section affine metric}, on the assumption of spectral  rigidity on stable bundle, we endow an affine metric on each leaf of the lifting stable foliations, which will be useful for the proof of sufficient  parts  of our theorems.
	In section \ref{special}, we prove the "sufficient" parts of Theorem \ref{main theorem 1} and Theorem \ref{main theorem 2}, which state that the rigidity of periodic stable Lyapunov spectrums implies the existence of integrable unstable bundle and we  also prove Corollary \ref{cor-1} and Corollary \ref{cor-d} in this section.
	
	\vskip 0.5 \baselineskip
	
	\noindent {\bf Acknowledgements:}  S. Gan was partially supported by NSFC (11831001, 12161141002) and National Key R\&D Program of China (2020YFE0204200). Y. Shi was partially supported by National Key R\&D Program of China (2021YFA1001900) and NSFC (12071007, 11831001, 12090015).
	
	\section{Preliminaries}\label{sec:preliminaries}
	
	For short, an  Anosov map $f:M\to M$  is called \textit{special},  if it has the integrable unstable bundle.
	\subsection{Global properties}
	
	For studying an Anosov map, one can lift it to the universal cover.  In fact,
	Ma$\rm \tilde{n} \acute{e}$ and Pugh proved the following proposition which allows us to observe the dynamics on the universal cover.
	
	\begin{proposition}[\cite{manepugh}]\label{lifting Anosov}
		Let $\tilde{M}$ be the universal cover of $M$ and $F:\tilde{M}\to \tilde{M}$ be a lift of $f:M\to M$.
		Then $f$ is an Anosov map if and only if $F$ is an Anosov diffeomorphism.
	\end{proposition}
	
	As usual, we define the \textit{stable manifolds} of the Anosov diffeomorphism $F$, denoted by $\tildeF^s(x)$,
	\begin{align}
		\tildeF^s(x):=\big\{y\in \tilde{M} \; \big|\;  d(F^k(y),F^k(x))\to 0 \;\;{\rm as}\;\; k\to +\infty \big\},\label{defineglobalmanifold}
	\end{align}
	for all $x\in \tilde{M}$,
	and the  \textit{unstable manifolds} $\tildeF^u(x)$ by iterating backward.  And we define the \textit{local (un)stable} manifolds with size $\delta$, denoted by  $\tildeF^{s/u}(x,\delta)$,
	\begin{align}
		\tildeF^{s/u}(x,\delta):=\big\{y\in \tildeF^{s/u}(x) \; \big|\;  d_{\tildeF^{s/u}}(x,y)\leq \delta \big\},  \label{definelocalmanifold}
	\end{align}
	for all $x\in \tilde{M}$,  where $d_{\tildeF^{s/u}}(\cdot,\cdot)$ is induced by the metric on $\tilde{M}$.
	\vspace{0.3cm}
	
	In the rest of this paper, we restrict the manifolds $M$ to be a $d$-torus $\mathbb{T}^d$.
	Let $f:\mathbb{T}^d\to \mathbb{T}^d$ be an Anosov map and $A:\mathbb{T}^d\to \mathbb{T}^d$  its linearization. The following proposition exhibits the equivalent condtion of	an Anosov map being conjugate with its linearization.
	
	\begin{proposition}[\cite{specialconjugate}]\label{special and conjugate}
		Let $f$ be an Anosov map on torus, then $f$ is conjugate to its linearization
		if and only if $f$ is special.
	\end{proposition}
	
	On the other hand, from the observation of Proposition \ref{lifting Anosov}, we can expect that the liftings of $f$ and $A$ are conjugate.  Let $F:\mathbb{R}^d \to \mathbb{R}^d$ be a lift of $f$ and $\overline{A}:\mathbb{R}^d \to \mathbb{R}^d$ be the lift of
	$A:\mathbb{T}^d\to \mathbb{T}^d$ induced by the same projection $\pi:\mathbb{R}^d\to \mathbb{T}^d$.
	It means that  $\pi\circ F= f\circ \pi$ and $\pi\circ \overline{A}= A\circ \pi$.
	For short, we denote $\overline{A}$ by $A$ if there is no confusion.
	The following proposition \cite[Proposition 8.2.1 and Proposition 8.4.2]{aoki} says that we do have a conjugacy between $F$ and $A$.
	
	\begin{proposition}[\cite{aoki}]\label{lifting conjugate}
		Let $f:\TT^d\to\TT^d$ be an Anosov map with lifting $F:\RR^d\to \RR^d$ and $A$ be its linearization.  There is a unique bijection $H:\mathbb{R}^d\to \mathbb{R}^d$ such that
		\begin{enumerate}
			\item $A\circ H=H\circ F$.
			\item $H$ and $H^{-1}$ are both uniformly continuous.
			\item There exists $C>0$ such that $\|H-Id\|<C$ and $\|H^{-1}-Id\|<C$.
		\end{enumerate}
	\end{proposition}
	
	\begin{remark}
		Without losing generality, we can always assume that $F(0)=0$ and $H(0)=0$.
	\end{remark}	
	
	By proposition \ref{special and conjugate}, it is clear that $f$ is special if and only if $H$ is commutative with $\ZZ^d$-action, namely, $$H(x+n)=H(x)+n,\quad \forall x\in\RR^d \;\;\;{\rm and}\;\;\; \forall n\in\ZZ^d.$$
	Although in general $H$ cannot be commutative with $\ZZ^d$-action, we will see it can be commutative with $\ZZ^d$-action as a stable leaf conjugacy.
	
	\begin{notation*}
		Denote the stable/unstable bundles and foliations of $A$ on $\TT^d$ by $L^{s/u}$, $\mathcal{L}^{s/u}$ and on $\RR^d$ by
		$\tilde{L}^{s/u}$, $\tildeL^{s/u}$ respectively.
	\end{notation*}
	It is clear that $H$ is a stable/unstable \textit{leaf conjugacy} between $F$ and $A$, namely,
	\begin{align}
		H(\tildeF^{s/u})=\tildeL^{s/u} \quad {\rm and} \quad H(\tildeF^{s/u}(Fx))= A(\tildeL^{s/u}(Hx)). \label{Hpreservestable}
	\end{align}
	Indeed, by the topological character of  stable/unstable  foliations \eqref{defineglobalmanifold} for $A$ and $F$, one can get \eqref{Hpreservestable}, directly.
	Especially, we mention that $\tildeF^s$ and $\tildeF^u$ admit the \textit{Global Product Structure}, namely, any two leaves $\tildeF^s(x)$ and $\tildeF^u(y)$ intersect transversely at a unique point in $\RR^d$.
	
	\begin{proposition}\label{H is Zd restricted on stable}
		Assume that $H$ is given by Proposition \ref{lifting conjugate}. Then for every $x\in\RR^d$ and $n\in\ZZ^d$,
		$$H(x+n)-n\in \tildeL^s\big(H(x)\big) \quad {\rm and } \quad H^{-1}(x+n)-n\in \tildeF^s\big(H^{-1}(x)\big).$$
	\end{proposition}
	\begin{proof}
		By Proposition \ref{lifting conjugate}, let $H:\mathbb{R}^d\to \mathbb{R}^d$
		be the unique conjugacy with $\|H^{-1}-Id\|<C_0$.
		Now, we iterate these two points $H^{-1}(x+n)-n$ and $H^{-1}(x)$  forward by $F$. Note that since $f$ and $A$ are homotopic, we have $F^k(x+n)=F^k(x)+A^kn$, for all $x\in \mathbb{R}^d$, $n\in \mathbb{Z}^d$ and $k\in \NN$.
		It follows that
		\begin{align*}
			\qquad \qquad \qquad \qquad
			&\quad d\left(F^k\circ H^{-1}(x)\;, \; F^k (H^{-1}(x+n)-n)\right),\\
			&=d\left(H^{-1}(A^kx)+A^kn \;,\; H^{-1}(A^kx+A^kn))\right),\\
			&\le d\left(H^{-1}(A^kx)+A^kn \;,\; A^kx+A^kn\right)\;+\;d\left(A^kx+A^kn \;,\; H^{-1}(A^kx+A^kn))\right),\\
			& = d\left(H^{-1}(A^kx)\;,\; A^kx\right)\;+\; d\left(A^kx+A^kn \;,\; H^{-1}(A^kx+A^kn))\right)
			\le 2C_0.
		\end{align*}
		Let $k$ tend to infinity, the fact that $d\big(F^k(H^{-1}(x)+n),F^k (H^{-1}(x+n))\big)$ is always bounded by a uniform constant is sufficient to prove $H^{-1}(x+n)-n\in \tildeF^s\big(H^{-1}(x)\big)$.
		
		The proof for $H(x+n)-n\in \tildeL^s\big(H(x)\big)$ deduces from the fact that $H(\tildeF^s)=\tildeL^s$.
	\end{proof}
	
		The following three propositions are all related to approching by "special" $\ZZ^d$-sequences which will be useful in Section \ref{special}.
	
	\begin{proposition}\label{nmH}
		Let $f:\TT^d\to\TT^d$ be an Anosov map with linearization $A$, and $H$ be the conjugacy between its lifting $F$ and $A$.  There exist $C>0$ and  $\{\varepsilon_m\}$ with $\varepsilon_m\to 0$ as $m\to +\infty$, such that for every $x\in\RR^d$ and every $n_m\in \ZZ^d$ satisfying
		$$
		A^{-i}n_m \in \ZZ^d, \qquad \forall 1\leq i\leq m,
		$$
		the following two inequations hold
		$$|H(x+n_m)-H(x)-n_m|<C\cdot \| A|_{L^s}\|^m,$$
		and
		$$|H^{-1}(x+n_m)-H^{-1}(x)-n_m|<\varepsilon_m.$$
	\end{proposition}
	\begin{proof}
		By $A^{-i}n_m \in \ZZ^d (1\leq i\leq m)$, we have $F^{-i}(x+n_m)-F^{-i}(x)=A^{-i}n_m\in\ZZ^d$, for all $1\leq i\leq m$.
		By $|H-id|<C_0$, we have $$d\left(H(F^{-m}x+A^{-m}n_m)\;,\;H(F^{-m}x)+A^{-m}n_m\right) \leq 2C_0.$$
		By Proposition \ref{H is Zd restricted on stable}, one has $H(x)+n\in \tildeL^s\big(H(x+n)\big)$, for every $x\in\RR^d$ and $n\in \ZZ^d$.
	    Hence,
		$$d\Big(A^m\left(H(F^{-m}x+A^{-m}n_m)\right) \;,\;  A^m\left(H(F^{-m}x)+A^{-m}n_m\right) \Big) \leq 2C_0\cdot \| A|_{L^s}\|^m .$$
		That is
		$$|H(x+n_m)-H(x)-n_m|<2C_0\cdot \| A|_{L^s}\|^m.$$
		On the other hand, by the uniform continuity of $H^{-1}$, there exists $\varepsilon_m$ satisfying $\varepsilon_m\to 0$ as $m\to +\infty$ such that if  $d(x,y)<2C_0\cdot \| A|_{L^s}\|^m$, then $d(H^{-1}(x),H^{-1}(y))<\varepsilon_m$.
		Thus,
		$$d\left(x+n_m \;,\; H^{-1}(H(x)+n_m) \right)<\varepsilon_m.$$
		Equivalently, for every $x\in\RR^d$, we have
		$$\Big|H^{-1}(x+n_m)-H^{-1}(x)-n_m\Big|<\varepsilon_m.$$
	\end{proof}
	
	The following proposition  is a corollary of Proposition \ref{nmH}. It says that although the $F$-invariant foliation may not be commutative with $\ZZ^d$-actions, it can "almost" be commutative with $n_m$-actions if $H$ maps it to an $A$-invariant linear foliation.
	\begin{proposition}\label{nmF}
	 On the assumption of Proposition \ref{nmH} and assume that  $\tildeL$  is an $A$-invariant linear foliation and the sequence $\{n_m\}_{m\in\NN}\subset \ZZ^d$ satisfies $n_m \in A^m\ZZ^d $. If each leaf of the $F$-invariant foliation $\tildeF:=H^{-1}(\tildeL)$ is $C^1$-smooth, then for every $x\in\RR^d$ and every $R>0$,
		$$\tildeF(x+n_m,R)-n_m\to \tildeF(x,R),$$
		as $m\to +\infty$, where the local manifold  $\tildeF(x,R):=\{ y\in\tildeF(x)\; |\;d_{\tildeF} (x,y)\leq R\}$.
	\end{proposition}
	
	\begin{proof}
		By Proposition \ref{nmH}, when $m\to +\infty$,
		\begin{align}
			d_H\Big( H\left(  \tildeF(x+n_m,R)-n_m \right) \; \;,\; \;   H\left(  \tildeF(x+n_m,R) \right)-n_m   \Big)\to 0, \label{prop2.7.1}
		\end{align}
		where $d_H(\cdot,\cdot)$ is Hausdorff distance.  Note that $\mathcal{T}:=H\left(  \tildeF(x+n_m,R) \right)$ is a local leaf on $\tildeL\big( H(x+n_m)\big)$. Since $\tildeL$ is commutative with $\ZZ^d$-actions, the set $(\mathcal{T}-n_m)$ is a copy of $\mathcal{T}$ on $\tildeL\big( H(x+n_m)-n_m\big)$.	Again, by Proposition \ref{nmH}, as $m\to +\infty$,
		\begin{align}
			d_H\big(\mathcal{T}-n_m\; \;,\;\;  \mathcal{T}-H(x+n_m)+H(x) \big)\to 0. \label{prop2.7.2}
		\end{align}
		Then, combining \eqref{prop2.7.1} and  \eqref{prop2.7.2},  one has
		$$d_H\Big(  H\left(  \tildeF(x+n_m,R)-n_m \right)  \; \;,\;\;  \mathcal{T}-H(x+n_m)+H(x)\Big)\to 0.$$
		By the uniform continuity of $H^{-1}$, it follows that
		$$d_H\Big(   \tildeF(x+n_m,R)-n_m  \; \;,\;\; H^{-1}\big( \mathcal{T}-H(x+n_m)+H(x)\big)\Big)\to 0.$$
		Since $\big( \mathcal{T}-H(x+n_m)+H(x)\big)$ is  a copy of $\mathcal{T}$ on $\tildeL\big( H(x)\big)$, the set $H^{-1}\big( \mathcal{T}-H(x+n_m)+H(x)\big)$ is a local leaf on $\tildeF(x)$. Moreover, its size tends to $R$ as $m\to+\infty$.
	\end{proof}
	\begin{remark}
		The foliation $\tildeF$ in Proposition \ref{nmF} can be the unstable foliation or the center foliation of $F$, where $F:\RR^d\to\RR^d$ is a lifting of an Anosov map on torus.
	\end{remark}
	
		A foliation $\mathcal{F}$ on $\TT^d$ is called \textit{minimal}, if its every leaf is dense. It is clear that  if $A$ is irreducible, then every $A$-invariant linear foliation on $\TT^d$ is minimal (a complete proof is available in \cite{ganzhangren}). The following proposition actually says that  the projection of each leaf of $F$-invariant foliation onto $\TT^d$ is dense  if $H$ maps it to an $A$-invariant linear foliation.
	\begin{proposition}\label{projection leaf dense}
		  On the assumption of Proposition \ref{nmH} and assume that  $A$ is irreducible and $\tildeL$  is an $A$-invariant linear foliation.  Let $\tildeF=H^{-1}(\tildeL)$ be a foliation on $\RR^d$. Then for any $x,y\in\RR^d$,
		 \begin{enumerate}
		 	\item There exist $x_m\in\tildeL(x)$ and $n_m\in A^m \ZZ^d$ , such that $(x_m+n_m)\to y$ as $m\to +\infty$.
		 	\item There exist $z_m\in\tildeF(x)$ and $n_m\in A^m \ZZ^d$ , such that $(z_m+n_m)\to y$ as $m\to +\infty$.
		 \end{enumerate}
	\end{proposition}
	\begin{proof}
		Since $A$ is irreducible,  the set $\{\tildeL(0)+\ZZ^d\}$ is dense in $\RR^d$.  Fix $m\in \NN$, one has that the set $A^m\{\tildeL(0)+\ZZ^d \}$ is dense in $\RR^d$. It follows  that
		$\{\tildeL(0)+A^m\ZZ^d \}$ is dense in $\RR^d$, since $A^m\tildeL(0)=\tildeL(0)$. This complete the proof for the first item.
		
		Now, applying the first item for points $H(x)$ and $H(y)$,  we can take $n_m\in A^m\ZZ^d$ and $x_m'\in\tildeL\big(H(x)\big)$ with $x_m'+n_m\to H(y)$. Let $z_m=H^{-1}(x_m')\in \tildeF(x)$. By Proposition \ref{nmH}, one has $d\big(H(z_m+n_m), H(z_m)+n_m \big)\to 0 $ as $m\to +\infty$. By the uniform continuity of $H^{-1}$, we get $(z_m+n_m)\to y$.
	\end{proof}
	
	To end this subsection, we state a proposition about the density of preimage sets for an irreducible toral endomorphism (may not be Anosov) which will be used in Section \ref{s-rigidity}.
	\begin{proposition}\label{preimage dense}
		Let $A\in M_d(\mathbb{Z})\cap \rm{GL}$$_d(\mathbb{R})$ be irreducible over $\mathbb{Q}$. It induces a torus endomorphism $A:\mathbb{T}^d\to \mathbb{T}^d$.
		Then there exists $C>0$ such that for every $k\ge 1$ and every $x_0\in\mathbb{T}^d$, the $k$-preimage set of $x_0$
		\begin{align*}
			 \{x\in\mathbb{T}^d|A^k(x)=x_0 \}
		\end{align*}
		is $\;C|{\rm det}(A)|^{-k/d}$-dense in $\mathbb{T}^d$.	
	\end{proposition}
	\begin{proof}
		Consider the $d$-dimensional real vector space
		$$V=\text{span}_{\mathbb{R}}\{ I_d,A,...,A^{d-1}\}$$
		and the lattice $V_{\mathbb{Z}}:=V\cap M_d(\mathbb{Z})$ in $V$. We claim that nonzero matrices in $V_{\mathbb{Z}}$ are invertible in $GL_d(\RR)$.
		In fact, if $M\in V_{\mathbb{Z}}-\{0\}$, then there is a nonzero rational polynomial $m\in \mathbb{Q}[x]$ with deg$(m)\leq d-1$ such that $M=m(A)$. Note that $m $ and the characteristic polynomial $\chi$ of $A$ are coprime over $\mathbb{Q}$, since $A$ is irreducible. So there exist $g,h\in\mathbb{Q}[x]$ such that $mg+\chi h=1$. It follows that $m(A)g(A)=I_d$. So $M=m(A)$ is invertible.
		
		Fix $\mathcal{K}\subset V $ be a compact covex symmetric subset with vol$(\mathcal{K})\geq 2^d\text{vol}(V/V_{\mathbb{Z}})$,
		and let $$C=
        {\sqrt{d}}\cdot
        \sup_{K\in\mathcal{K},v\in\mathbb{R}^d,\|v\|=1} \|Kv\|.$$
		We prove that for every $k\geq 1$,
  the $k$-preimage set of every $x_0\in\mathbb{T}^d$
  is $C|\text{det}(A)|^{-k/d}$-dense in $\mathbb{T}^d$.
		It suffices to prove that $A^{-k}\mathbb{Z}^d$ is $C|\text{det}(A)|^{-k/d}$-dense in $\mathbb{R}^d$.
		
		Let $L:V\to V$ be the linear map defined be $L(X)=AX$. The matrix of $L$ relative to the basis $\{ I_d,A,...,A^{d-1}\}$ of $V$ is
		the companion matrix of $\chi$. So $\text{det}(L)=\text{det}(A)$. It follows that the compact convex symmetric set $|\text{det}(A)|^{-k/d} \cdot L^k(\mathcal{K})$ has volume
		\begin{align*}
		   {\rm vol}\left( |\text{det}(A)|^{-k/d} \cdot L^k(\mathcal{K})  \right) 
		   &=|\text{det}(A)|^{-k} \cdot \text{vol}\left(L^k(\mathcal{K})\right) \\
		   &=|\text{det}(A)|^{-k}|\text{det}(L)|^k
		   \cdot\text{vol}(\mathcal{K}) \\
		   &=\text{vol}(\mathcal{K})\geq 2^d\text{vol}(V/V_{\mathbb{Z}}).
		\end{align*}
		By Minkowski's convex body theorem, it contains a nonzero matrix in $V_{\mathbb{Z}}$.
		Namely, there exist $K\in\mathcal{K}$ and $M\in V_{\mathbb{Z}}-\{0\}$  such that $|\text{det}(A)|^{-k/d}A^kK=M$.
		Note that $M$ is invertible, so $K$ is also invertible. For $r>0$, let $B(r)\subset \mathbb{R}^d$ denote the open ball with radius $r$
		centered at the origin. Then
		\begin{align*}
			\mathbb{R}^d&=|\text{det}(A)|^{-k/d}K\mathbb{R}^d=|\text{det}(A)|^{-k/d} K\big(\mathbb{Z}^d+B(\sqrt{d})\big) \\
			&=A^{-k}M\mathbb{Z}^d+|\text{det}(A)|^{-k/d} KB(\sqrt{d}) \\
			&\subseteq A^{-k}\mathbb{Z}^d+B\big(C|\text{det}(A)|^{-k/d}\big).
		\end{align*}
		This means that $A^{-k}\mathbb{Z}^d$ is $C|\text{det}(A)|^{-k/d}$-dense in $\mathbb{R}^d$.
	\end{proof}

	\subsection{Dominated splitting on the inverse limit space}
	Now we introduce the dynamics on the inverse limit space.
	Note that the inverse limit space has compactness which the universal cover  lacks.
	
	Firstly,  we clear the definition of \textit{inverse limit space}. Let $(M,d)$ be a compact metric space and $M^{\mathbb{Z}}:=\{(x_i) \;|\;x_i\in M,  \forall i\in \mathbb{Z} \}$ be the product topological space.  $M^{\mathbb{Z}}$ is compact by Tychonoff theorem  and it can be metrizable by the metric
	$$
	\tilde{d}((x_i),(y_i))= \sum_{-\infty}^{+\infty} \frac{d(x_i,y_i)}{2^{|i|}}.
	$$
	Let $\sigma: M^{\mathbb{Z}}\to M^{\mathbb{Z}}$ be the \textit{(left) shift} homeomorphism by
	$(\sigma (x_i))_j=x_{j+1}$, for all $j\in \mathbb{Z}$.
	For a continuous map $f:M\to M$, the inverse limit space of $f$ is
	$$
	M_f:=\big\{(x_i) \;|\; x_i\in M \;\; {\rm and} \;\; f(x_i)=x_{i+1},  \forall i\in \mathbb{Z} \big\}.
	$$
	With the metric $\tilde{d}(\cdot,\cdot)$, the inverse limit space
	$(M_f,\tilde{d})$ is a closed subset of $(M^{\mathbb{Z}},\tilde{d})$. So it is a compact metric space.
	It is clear that $(M_f,\tilde{d})$ is $\sigma$-invariant.
	
	\begin{definition}[\cite{Przytycki}]\label{defprz}
		A $C^1$ local diffeomorphism $f:M\to M$ is called Anosov map, if there exist constants $C>0$
		and $0<\mu<1$ such that, for every $\tilde{x}=(x_i) \in M_f$,
		there exists a hyperbolic splitting
		\begin{align*}
			T_{x_i}M=E^s(x_i,\tilde{x})\oplus E^u(x_i,\tilde{x}),\quad  \forall i\in \mathbb{Z},
		\end{align*}
		which is $Df$-invariant
		\begin{align*}
			D_{x_i}f\left(E^{s}(x_i,\tilde{x})\right)
			=E^s(x_{i+1},\tilde{x})
			\qquad{\rm and}\qquad D_{x_i}f\left(E^{u}(x_i,\tilde{x})\right)
			=E^u(x_{i+1},\tilde{x}),
			\qquad\forall i\in\mathbb{Z},
		\end{align*}
		and for all $n>0$ the following estimates hold:
		\begin{align*}
			&\|D_{x_i}f^n (v)\|\le C\mu^n\|v\|, \qquad \quad \forall v\in E^s(x_i,\tilde{x}), \;\;\forall i\in\mathbb{Z},\\
			&\|D_{x_i}f^n (v)\|\ge C^{-1}\mu^{-n}\|v\|, \;\;\quad \forall v\in E^u(x_i,\tilde{x}),\;\;\forall i\in\mathbb{Z}.
		\end{align*}
	\end{definition}

	We extend the hyperbolic splitting on the inverse limit space to the dominated splitting case. We say a local diffeomorphism $g:M\to M$ admits a \textit{dominated splitting}
	$$TM=E_1\oplus E_2\oplus...\oplus E_m,$$
	if for every $i=1,\cdots,m$, each subbundle $E_i$ is $Dg$-invariant and there exist  $C>0$, $0< \lambda< 1$ such that for any $x\in M$  and any unit vectors $u\in E_i(x)$ and $v\in E_{i+1}(x)$,
	$$\| D_xg^nu \| \leq C\lambda^n \| D_xg^nv \|, \quad \forall n\in\NN.$$
	
	Let $A:\TT^d\to \TT^d$ be a linear Anosov map and admits the \textit{finest (on stable bundle) dominated  splitting},
	\begin{align}
		T\TT^d=L^s_1\oplus L^s_2\oplus...\oplus L^s_k\oplus L^u, \label{Acondition1}
	\end{align}
	where ${\rm dim}L^s_i=1$, $1\leq i\leq k$, such that
	\begin{align}
		0<|\mu^s_1(A)|<|\mu^s_2(A)|<...<|\mu^s_k(A)|<1<|\mu_{\rm min}^u(A)|,\label{Acondition2}
	\end{align}
	where $\mu^s_i(A)$ is the eigenvalue with respect to $L^s_i$ and $\mu_{\rm min}^u(A)=m(A|_{L^u})$ the mininorm of $A$ restricted on $L^u$. Denote by $\mu_{\rm max}^u(A)=\|A\|$, the norm of $A$. 	And denote by $\lambda^s_i(A)={\rm log}|\mu^s_i(A)|$ the stable Lyapunov exponent of the subbundle $L^s_i$.
	
	\begin{notation*}
		We use the following notations to denote the joint bundles of $A$ which are anological for $\tilde{L}$, $\mathcal{L}$, $\tildeL$ and the bundles and foliations of $f$ and $F$, if they are well defined.
		\begin{enumerate}
			\item $L^s_{(i,j)}:=L^s_i\oplus L^s_{i+1}\oplus ...\oplus L^s_j$, $1\leq i\leq j\leq k$.
			\item $L^{s,u}_{(i,j)}:=L^s_i\oplus L^s_{i+1}\oplus ...\oplus L^s_j\oplus L^u$, $1\leq i\leq j\leq k$.
			\item $L^s_{(i,i)}=L^s_i$,  $L^{s,u}_{(i,i)}=L^{s,u}_i=L^s_i\oplus L^u$, $1\leq i\leq k$.
		\end{enumerate}	
	\end{notation*}

	\begin{proposition}\label{s-bundle-continuous}
	Let $A\in M_d(\ZZ)\cap GL_d(\RR)$ induce a toral Anosov map with the finest (on stable bundle)	dominated splitting satisfying \eqref{Acondition1} and \eqref{Acondition2}. Then, there exists a $C^1$ neighborhood $\mathcal{U}\subset C^1(\TT^d)$ of $A$ such that for all $f\in\mathcal{U}$ and all $\tilde{x}=(x_n) \in \TT^d_f$,  there exists a dominated splitting
		$$T_{x_n}\TT^d=E^s_1(x_n,\tilde{x})\oplus...\oplus E^s_k(x_n,\tilde{x}) \oplus E^u(x_n,\tilde{x}),\quad  \forall n\in \ZZ,$$
		where dim$E^s_i=1, 1\leq i\leq k$.
		And the bundles $E^s_i(x_n,\tilde{x}) $ and $E^u(x_n,\tilde{x})$ are continuous with $\tilde{x}$. Especially, there exist bundles $E^s_{(1,i)} (1\leq i\leq k)$ defined on $\TT^d$ such that
		\begin{align}
			E^s_1\subset E^s_{(1,2)}\subset...\subset E^s_{(1,k-1)}\subset E^s_{(1,k)}=E^s. \label{strongstableexistsonTd}
		\end{align}
	
		Moreover, for any $\alpha\in(0,\frac{\pi}{2})$ and $\delta>0$, there exists $\mathcal{U}\subset C^1(\TT^d)$ such that for every $f\in\mathcal{U}$ and $\tilde{x}=(x_n)\in\TT^d_f$ and  $1\leq i\leq k$,
		\begin{align}
				\angle\left( E^s_i(x_0,\tilde{x}),  L^s_i(x_0) \right)\leq \alpha \quad
			{\rm and} \quad \|Df|_{E^s_i(x_0,\tilde{x})}\| \in \big[\mu^s_i(A)-\delta, \mu^s_i(A)+\delta\big].\qquad     \label{Acontrols}
		\end{align}
	    And,
		\begin{align}
			\angle\left( E^u(x_0,\tilde{x}),  L^u(x_0) \right)\leq \alpha \quad
			{\rm and} \quad \|Df|_{E^u(x_0,\tilde{x})}\| \in \big[\mu^u_{\rm max}(A)-\delta, \mu_{\rm max}^u(A)+\delta\big].      \label{Acontrolu}
		\end{align}
		
	\end{proposition}
	
	\begin{proof}
		We use invariant cone-fields to complete this proof. One may find more details in \cite[Appendix B.1]{BDV} and  \cite[Section 2.2 ]{PotrieCrovisierbook}.
		
		Let $S\oplus U$ be a dominated splitting on $T\TT^d$ of $A$ and $\mathcal{C}_{\alpha}$ be an $A$-invariant cone-fields contain $U$ with size $\alpha$. That is
		$$\mathcal{C}_{\alpha}(x)=\big\{  \nu=\nu_S+\nu_U\in  T_x\TT^d: \;\; \| \nu_S \|\leq \alpha \|\nu_U\|   \big\} \quad {\rm and} \quad
		DA(\mathcal{C}_{\alpha}(x))\subset \mathcal{C}_{\mu\cdot\alpha}(x),$$
		for some $\mu<1$.
		So, there exists a $C^1$-neighborhood  $\mathcal{U}$ of $A$ such that for any $f\in\mathcal{U}$ and any $x\in\TT^d$,
		one has $Df(\mathcal{C}_{\alpha}(x)) \subset {\rm int}\left(    \mathcal{C}_{\alpha}(fx)\right)$. This implies the existence of dominated splitting (see  \cite[Section 2.2 ]{PotrieCrovisierbook}). We mention that we can make a larger perturbation as long as $\mathcal{C}_{\alpha}$ can be contracted into itself by finite iterations of $Df$.
		And let $\mathcal{C}^*_{\alpha}$ be the closure of the complement of $\mathcal{C}_{\alpha}$. It is clear that $\mathcal{C}^*_{\alpha}$ is also a cone-field  which is contracted by $Df^{-1}$, if the preimage is given.
		Fix an $f$-orbit $\tilde{x}=(x_n)\in \TT^d_f$ and let
		$$S_f(x_0,\tilde{x}):= \bigcap_{n\leq0} Df^{-n}\big(\mathcal{C}^*_{\alpha}(x_n)\big),\quad U_f(x_0,\tilde{x}):= \bigcap_{n\leq0} Df^{n}\big(\mathcal{C}_{\alpha}(x_{-n})\big).$$
		Then $S_f\oplus U_f$ give a dominated splitting on the inverse limits space $\TT^d_f$. Note that $S_f(x_0,\tilde{x})$ is independent of the choice of orbits for $x_0$.
		
		For finding out $E_i^s(x_0,\tilde{x})$, we consider these two dominated splittings for $A$
		$$L^s_{(1,i)}\oplus L^{s,u}_{(i+1,k)} \quad {\rm and} \quad L^s_{(1,i-1)}\oplus L^{s,u}_{(i,k)}.$$
		So, we can get the following two dominated splittings under the $C^1$-perturbation,
		$$E^s_{(1,i)}(x_0,\tilde{x}) \oplus E^{s,u}_{(i+1,k)}(x_0,\tilde{x})
		\quad {\rm and} \quad E^s_{(1,i-1)}(x_0,\tilde{x}) \oplus E^{s,u}_{(i,k)}(x_0,\tilde{x}).$$
		Hence, we get  $E_i^s(x_0,\tilde{x})= E^s_{(1,i)}(x_0,\tilde{x})\cap E^{s,u}_{(i,k)}(x_0,\tilde{x})$.
    	It is clear that  this dominated splitting is continuous with respect to orbits.  Meanwhile, $E^s_{(1,i)}(x_0,\tilde{x})$ only depends on $x_0$. It follows that the bundle $E^s_{(1,i)} $ is well defined on $\TT^d$. The previous proof also allows the control \eqref{Acontrols} and \eqref{Acontrolu} for bundles of $f$.
	\end{proof}

	\subsection{Foliations on the universal cover}\label{subsec: foliation on Rd}
	In this subsection, we always assume that $A\in M_d(\ZZ)\cap GL_d(\RR)$ admits the finest (on stable bundle) dominated splitting (see \eqref{Acondition1} and \eqref{Acondition2}).
	It is clear that there exists a $C^1$ neighborhood $\mathcal{U}$ of $A$ such that for every $f\in\mathcal{U}$, its lifting $F:\RR^d\to\RR^d$ admits dominated splitting
	$$
	T\RR^d=
	\tilde{E}^s_1\oplus\tilde{E}^s_2\oplus ...\oplus \tilde{E}^s_k\oplus \tilde{E}^u.
	$$
	Let $H:\RR^d\to \RR^d$ be the unique conjugacy between $F$ and $A$ guaranteed by Proposition \ref{lifting conjugate}.
	
	Let $\pi:\RR^d\to \TT^d$ be the natural projection such that $\pi\circ F= f\circ \pi$.
	Note that any $f$-orbit $\tilde{x}\in \TT^d_f$ can be approached by $F$-orbits.
	Hence for every $1\leq i\leq j\leq k$, one can get
	\begin{align}
		\bigcup_{\tilde{x}\in\TT^d_f, \tilde{x}_0=x_0 } E^s_{(i,j)}(x_0,\tilde{x})=\overline{ \bigcup_{\pi(y)=x_0} D_y\pi \left( \tilde{E}^s_{(i,j)}(y) \right)},
		\qquad \forall\tilde{x}=(x_n)\in \TT^d_f.
		\label{bundleonorbitandlift}
	\end{align}
	It is similar to $E^u(x_0,\tilde{x})$. We refer to \cite[Proposition 2.5]{tahzibi} for more details.
	This projection allows us to get properties of $F$ from ones of $\TT^d_f$. Especially, we have the following remarks.
	
	\begin{remark}\label{size of U}
		By Proposition \ref{s-bundle-continuous}, we can choose $\mathcal{U}\subset C^1(\TT^d)$ small enough such that
		\begin{align}
			\mu^s_i(A) -\frac{\varepsilon_0}{3} \leq \|DF|_{\tilde{E}^s_i(x)} \| \leq \mu^s_i(A) +\frac{\varepsilon_0}{3} \quad
			{\rm and} \quad   \angle\left(  \tilde{E}^s_i(x), \tilde{L}^s_i(x)    \right)<\alpha <\frac{\pi}{2},       \label{epsilon0}
		\end{align}
		for every $x\in\RR^d$ and $1\leq i\leq k$, where $\varepsilon_0={\rm min}\left\{ 1-\mu_k^s(A), \mu^s_{i+1}(A)-\mu^s_i(A):  1\leq i\leq k-1  \right\}$.
	\end{remark}
\vspace{2mm}
   \begin{remark}\label{holderbundle}
	The distributions $\tilde{E}^s_{(i,j)} (1\leq i\leq j\leq k)$ and $\tilde{E}^u$ are all H$\ddot{\rm o}$lder continuous on $T\RR^d$. One can get this from \cite[Theorem 2.3]{pesinbook} with the fact that the angle $\angle(\tilde{E}_1,\tilde{E}_2)$ is uniformly away from $0$, where $\tilde{E}_1, \tilde{E}_2 $ are two different bundles in $\big\{  \tilde{E}^u,\; \tilde{E}^s_i: \;1\leq i\leq k  \big\}$ (also see Remark \ref{uniform continuity of foliation on Rd}).
    \end{remark}
	\vspace{0.3cm}
	We say $\tildeF$ is a \emph{quasi-isometric} foliation on $\RR^d$, if there exist contants $a,b>0$ such that
	\begin{align}
		d_{\tildeF}(x,y)\leq a\cdot d(x,y)+b,\quad \forall x\in \RR^d, y\in\tildeF(x). \label{defquasi-isometric}
	\end{align}
	 A foliation $\tildeF$ defined on $\RR^d$ is called \emph{$\ZZ^d$-periodic} (equivalently,  \emph{commutative with $\ZZ^d$-actions}), if
	 $$
	 \tildeF(x+n)=\tildeF(x)+n,
	 \qquad \forall x\in\RR^d,
	 \quad\forall n\in\ZZ^d.
	 $$
	Similarly, a \emph{$\ZZ^d$-periodic} bundle $\tilde{E}$ on $T\RR^d$  means that
	$$
	\tilde{E}(x+n)=DT_n\left(\tilde{E}(x)\right),
	\qquad \forall x\in\RR^d, \quad\forall n\in\ZZ^d.
	$$
	Here $T_n:\RR^d\to\RR^d$ is the translation $T_n(x)=x+n$.
	
	\begin{proposition}\label{leaf proposition}
		There exists a $C^1$ neighborhood $\mathcal{U}\subset C^1(\TT^d)$  of $A$ such that for every $f\in\mathcal{U}$ and its lifting $F$, we have the followings,
		\begin{enumerate}
			\item  The $F$-invariant bundles $\tilde{E}^s_{(i,j)} (1\leq i\leq j\leq k)$ and $\tilde{E}^u$ are uniquely integrable.  Denote the integral foliations by $\tildeF^s_{(i,j)}$, $\tildeF^u$ and $\tildeF^s_i$ ( if $j=i$).  Moreover, the foliation $\tildeF^s_{(i_0,j_0)}$ is subfoliated by $\tildeF^s_{(i,j)}$, for any $1\leq i_0\leq i \leq j \leq j_0 \leq k$.
			\item The strong stable bundle $\tilde{E}^s_{(1,i)} $ and strong stable foliation $\tildeF^s_{(1,i)}$ $(1\leq i\leq k)$ are both $\ZZ^d$-periodic.
			\item The foliation $\tildeF^s_{(i,j)} (1\leq i\leq j\leq k)$ is quasi-isometric. Especially, for every Anosov map on torus with one-dimensional stable bundle, the lifting of stable foliation is quasi-isometric.
			\item $H$ preserves the weak stable foliations. It means that $H(\tildeF^s_{(i,k)})=\tildeL^s_{(i,k)}$,  for all $1\leq i\leq k$.
		\end{enumerate}
	\end{proposition}
	
	\begin{proof}
		Let  $\mathcal{U}$ be given by Remark \ref{size of U}. Consider the dominated splitting
		$$\tilde{E}^s_{(1,i-1)}\oplus \tilde{E}^s_{(i,k)}\oplus \tilde{E}^u,$$
		where $2\leq i\leq k$. By the Stable Manifold Theorem e.g.  \cite[Theorem 4.1 and Theorem 4.8]{pesinbook}, we have that
		$\tilde{E}^s_{(1,i-1)}$ and $\tilde{E}^u$ are always integrable. Although \cite{pesinbook} prove it for diffeomorphism, in our case Remark \ref{size of U} and Remark \ref{holderbundle} provides the uniform continuity and domination of bundles to replace the compactness. Moreover, by Proposition \ref{s-bundle-continuous} and \eqref{bundleonorbitandlift}, the bundle $\tilde{E}^s_{(1,i)} $ and foliation $\tildeF^s_{(1,i)}$ $(1\leq i\leq k)$ are $\ZZ^d$-periodic.
		
		For the integrability of $\tilde{E}^s_{(i,k)}$,  in the case of diffeomorphism,  the linear Anosov system $A:\RR^d\to \RR^d$ is robustly dynamically coherent (see \cite[Theorem 7.6]{HPS} also \cite[Proposition 3.2]{PughShub}) and this also holds for our case by the same reason of integrability for strong stable bundles. Hence, we have that
		 $$
		\tildeF^s_{(i,j)}=\tildeF^s_{(1,j)}\cap\tildeF^s_{(i,k)}
		$$
		is a foliation tangent to $\tilde{E}^s_{(i,j)}$. So $\tilde{E}^s_{(i,j)}$ is integrable for all $1\leq i\leq j\leq k$.  We refer to  \cite[Lemma 6.1]{Gogolevhighdimrigidity} for uniquely integrable property which is proved on the universal cover and also holds for non-invertible Anosov maps.
		It is clear that $\tildeF^s_{(i_0,j_0)}$ is subfoliated by $\tildeF^s_{(i,j)}$ for any $1\leq i_0\leq i \leq j \leq j_0 \leq k$. Moreover, $\tildeF^s_{(i_0,i)}$ and $\tildeF^s_{(i+1,j_0)}$ admit the Global Product Structure  on $\tildeF^s_{(i_0,j_0)}$.
		
		For a small perturbation,  in particular from \eqref{epsilon0}, we have that the foliation $\tildeF^s_{(i,j)}$ is uniformly transverse to $\tildeL^s_{(1,i-1)}\oplus\tildeL^s_{(j+1,k)}\oplus\tildeL^u$. By \cite[Proposition 4 ]{brinquasiisometric}, we have the quasi-isometric property for $\tildeF^s_{(i,j)}$. We mention that the proof of this actually need $\tildeF^s_{(i,j)}$  has a uniformly transverse plane and it is uniformly continuous (see Remark \ref{uniform continuity of foliation on Rd}).
		
    	For the case of dim$E^s=1$, since $|H-Id|$ is bounded, the unstable foliation $\tildeF^u$ is uniformly bounded by $\tildeL^u$. Namely, there exists $C_0>0$ such that $ \tildeF^u(x)$ and $\tildeL^u(x)$ are contained in the $C_0$-neighborhoods of each other,  for all $x\in\RR^d$. Fix  $n_0\in\ZZ^d$ such that $d(x,x+n_0)\geq 3C_0$, for all $x\in\RR^d$. It follows that the Hausdorff distance between $\tildeF^u(x)$ and $\tildeF^u(x+n_0)$ is bigger than $C_0$. Since the stable foliation $\tildeF^s$ and the unstable foliation $\tildeF^u$ admit the Global Product Structure, there exists $L_0$ such that for any $x\in\RR^d$ with $L(x)\leq L_0$, one has  $\tildeF^s(x,L(x))$ intersects $\tildeF^u(x+n_0)$ exactly once and the distance between $x$ and the intersection is bigger than $C_0$. Thus the one -dimensional foliation $\tildeF^s$ is always quasi-isometric, whether $f$ is a small perturbation or not. We refer readers to \cite{BBI2009} for more details.

		Finally, we prove that $H$ preserves the weak stable foliations.
		Let $y\in\tildeF^s(x)$ and we always have $H(y)\in\tildeL^s\big(H(x)\big)$. Note that $H(y)\in\tildeL^s_{(i,k)}\big(H(x)\big)$ if and only if
		$$d\left(  A^{-n}(Hy),  A^{-n}(Hx)   \right) \leq \left(\mu^s_i(A)\right)^{-n} \cdot d \left(Hy,Hx\right), \quad \forall n\in\NN.$$
		By Proposition \ref{lifting conjugate}, let $|H-id|<C_0$. One has that $H(y)\in\tildeL^s_{(i,k)}\big(H(x)\big)$ if and only if
		\begin{align}
			d\left(  F^{-n}(y),  F^{-n}(x)   \right) \leq \left(\mu^s_i(A)\right)^{-n}\cdot  d \left(Hy,Hx\right) +2C_0, \quad \forall n\in\NN. \label{leafproposition.1}
		\end{align}
		It implies that $H(y)\in\tildeL^s_{(i,k)}(H(x))$ if and only if $y\in\tildeF^s_{(i,k)}(x)$.
		Indeed, if $y\notin \tildeF^s_{(i,k)}(x)$, then there exists the unique point $z\in \tildeF^s_{(1,i-1)}(x)\cap \tildeF^s_{(i,k)}(y)$ with $a:=d_{\tildeF^s_{(1,i-1)}}(x,z)>0$. Let $b=d_{\tildeF^s_{(i,k)}}(z,y)$, note that $b$ may be zero.
		For $n\in\NN$ big enough, one has
		\begin{align}
			d(F^{-n}y,F^{-n}z)\leq d_{\tildeF^s_{(i,k)}}(F^{-n}y,F^{-n}z) \leq \left(\mu^s_i(A)-\frac{\varepsilon_0}{2}\right)^{-n}\cdot b, \label{leafproposition.2}
		\end{align}
		where $\varepsilon_0$ is given by \eqref{epsilon0} and
		$$d_{\tildeF^s_{(1,i-1)}}(F^{-n}x,F^{-n}z) \geq \left(\mu^s_{i-1}(A)+\frac{\varepsilon_0}{2}\right)^{-n}\cdot a.$$
		Since $\tildeF^s_{(1,i-1)}$ is quasi-isometric, there exists $0<C_1<1$ such that
		\begin{align}
			d(F^{-n}x,F^{-n}z)\geq C_1   \left(\mu^s_{i-1}(A)+\frac{\varepsilon_0}{2}\right)^{-n}\cdot a. \label{leafproposition.3}
		\end{align}
		Hence by \eqref{leafproposition.2} and \eqref{leafproposition.3}, one has
		$$d(F^{-n}x,F^{-n}y)\geq C_1   \left(\mu^s_{i-1}(A)+\frac{\varepsilon_0}{2}\right)^{-n}\cdot a-
		\left(\mu^s_i(A)-\frac{\varepsilon_0}{2}\right)^{-n} \cdot b,$$
		which contradicts with \eqref{leafproposition.1}.
	\end{proof}

	\begin{remark}\label{uniform continuity of foliation on Rd}
		We state the \textit{ uniform continuity of foliation} as follow.
		For given $\tildeF^s_i, (1\leq i\leq k)$ and two constants $C>0$,
		there exists $ \delta>0$ such that for every $x\in\RR^d$ and $y\in\tildeF^s_i(x)$ with $d_{\tildeF^s_i}(x,y)>C$,
		we have $d(x,y)>\delta$.
		Just note that, by the choice of the neighborhood $\mathcal{U}$, the angle $\angle (\tilde{E}^s_i, \tilde{L}^s_i)$ is uniformly bounded by $\alpha$. In fact, for any Anosov map (may not be a small perturbation) with a dominated splitting along orbit, $T_{x_n}\TT^d=E^s_1(x_n,\tilde{x})\oplus...\oplus E^s_k(x_n,\tilde{x}) \oplus E^u(x_n,\tilde{x})$  the angle between any two distinct subbundles is uniformly away from $0$, where dim$E^s_i$ may bigger than one.
	\end{remark}
	
	The following proposition says that the same periodic Lyapunov exponent implies it coincides with one of the linearization on the assumption that $H$ preserves the corresponding foliation.
\begin{proposition}\label{s-exp coincide with linear}
	Let $f\in\mathcal{U}$ given by Proposition \ref{leaf proposition}. Fix $1\leq i\leq k$ and suppose that  $H(\tildeF^s_i)=\tildeL^s_i$ and $\lambda^s_i(p,f)=\lambda^s_i(q,f)$ for every $p,q\in {\rm Per}(f)$.
	Then $\lambda^s_i(p,f)=\lambda^s_i(A)$, for all $p\in{\rm Per}(f)$.
	Especially, for every Anosov map $f$ on torus with dim$E^s=1$, if  $\lambda^s(p,f)=\lambda^s(q,f)$ for every $p,q\in {\rm Per}(f)$, then  $\lambda^s_i(p,f)=\lambda^s_i(A)$, for all $p\in{\rm Per}(f)$.
\end{proposition}	

\begin{proof}[Proof of Proposition \ref{s-exp coincide with linear}]\label{proof of s-exp}
		Since $\lambda^s_i(p,f)=\lambda^s_i(q,f)$ for all $p,q\in \rm{Per}$$(f)$,  $\mu:={\rm exp}\big(\lambda^s_i(f,p)\big)$ is a constant. We claim that there exists an adapted metric on $\RR^d$.
		
	\begin{claim}\label{adapted metric on Rd}
			For any $\delta>0$, there exists a smooth adapted Riemannian metric on $T\RR^d$ such that
			$$
			\mu\cdot (1+\delta)^{-1} <\|DF|_{\tilde{E}_i^s(x)}\|<\mu\cdot(1+\delta), \quad\forall x\in\RR^d.
			$$
		\end{claim}
	
	\begin{proof}[Proof of Claim \ref{adapted metric on Rd}]
	Since we have the Shadowing Lemma for Anosov maps (see \cite{aoki}), it can be proved as the existence of adapted  metrics for Anosov diffeomorphisms. For the convenience of readers, we prove it as follow.

Fix in advance a Riemannian metric on $T\RR^d$ which induces a norm $|\cdot|$. Let $$\mu_+:=\sup_{x\in\RR^d}|DF|_{\tilde{E}^s_i(x)}| \quad {\rm and } \quad \mu_-:=\inf_{x\in\RR^d}|DF|_{\tilde{E}^s_i(x)}|.$$
	By 	\eqref{bundleonorbitandlift} and the compactness of $\TT^d_f$, one has $\mu_+<+\infty$ and $\mu_->0$. Moreover, since $E^s_i(x_0,\tilde{x})$ is continuous with respect to $\tilde{x}=(x_i)\in\TT^d$, for any $\delta>0$, there exists $\alpha>0$ such that
		\begin{align}
		(1+\delta/2)^{-1}\le \frac{|DF|_{\tilde{E}^s_i(x)}|}{|DF|_{\tilde{E}^s_i(y)}|}\le 1+\delta/2,\label{adapted metric 1}
		\end{align}
	for any $x,y\in\RR^d$ with $d(x,y)<\alpha$.
	
	By the Shadowing Lemma, for given $\alpha>0$, there exists $\beta>0$ such that each periodic $\beta$-pseudo-orbit in $\TT^d$ can be $\alpha$-shadowing by a periodic orbit. Let $B_1,...,B_{n(\beta)}$ be finite many open $\beta$-balls  cover $\TT^d$. Since $f$ is transtive (also see Section \ref{section affine metric}), there exists $N_1\in\NN$ such that for any $B_i$ and $B_j$, $B_i$ can intersect $B_j$ within $N_1$-times iteration by $f$.
	
	Let $\pi:\RR^d\to \TT^d$ be the natural projection. For any $x\in\RR^d$ and $N_0\in\NN$ with $x_0:=\pi(x)\in B_{i_0}$ and $f^{N_0}(x_0)\in B_{i_1}$, there exists $y_0\in B_{i_1}$  and $N_2\in[0,N_1]$ such that $f^{N_2}y_0\in B_{i_0}$. It follows that
	$$\big\{x_0, f(x_0), ..., f^{N_0-1}(x_0), y_0,f(y_0),...,f^{N_2-1}y_0\big\}$$ is a periodic $\beta$-pseudo-orbit and is  $\alpha$-shadowing by a periodic orbit $p_0\in\TT^d$ with period $N(p_0):=N_0+N_2$.  Hence, there exists $p\in\pi^{-1}(p_0)$  such that $d\left(\pi(F^jx),\pi(F^jp)\right)<\alpha$, for all $j\in[0,N_0]$.
	
	Now, let  $N=N_0+N_1$ and
	$$\|\nu\|_N:= \prod_{n=0}^{N-1} |DF^n|_{\tilde{E}^s_i(x)}\nu |^{\frac{1}{N}}, \quad \forall \nu\in \tilde{E}^s_i(x).$$
	Note that $|DF^{N(p_0)}|_{\tilde{E}^s_i(p)}|= \mu^{N(p_0)}$, one has
	\begin{align}
		|DF^{N_0}|_{\tilde{E}^s_i(p)}|\in \left[ \mu^{N_0} \cdot \left( \frac{\mu}{\mu_+} \right)^{N_2}\;,\;   \mu^{N_0} \cdot \left( \frac{\mu}{\mu_-} \right)^{N_2}   \right]. \label{adapted metric 2}
	\end{align}
    Taking $\nu\in \tilde{E}^s_i(x)-\{0\}$,  by \eqref{adapted metric 1} and \eqref{adapted metric 2}, we calculate directly,
	\begin{align*}
		\frac{\| D_xF (\nu) \|_N}{\|\nu\|_N}
		&=\frac{| D_xF^N (\nu) |^{1/N}}{|\nu|^{1/N}}\\
		&=\left( \frac{\left| D_{F^{N_0}(x_0)} f^{N_1}\circ D_{x_0}F^{N_0} (\nu) \right|}{|\nu|} \right)^{1/N}\\
		&\in \left[\mu_-^{N_1/ N} \cdot \left( \frac{\mu}{1+\delta/2}  \right)^ {N_0/ N } \cdot  \left( \frac{\mu}{\mu_+} \right)^{N_2/N}\;,\;
		\mu_+^{N_1/ N} \cdot  \big( \mu \cdot (1+\delta/2 ) \big)^ {N_0/ N} \cdot  \left( \frac{\mu}{\mu_-} \right)^{N_2/N} \right].
	\end{align*}
	Hence, there exists $N_0$ big enough such that
	$$\frac{\| D_xF (\nu) \|_{N}}{\|\nu\|_{N}} \in \big(	\mu\cdot(1+\delta)^{-1}, 	\mu\cdot(1+\delta)    \big).$$
	There exists a norm $\|\cdot \|$ whose restriction on subbundle $\tilde{E}^s_i$ is $\|\cdot \|_{N}$ gives the smooth adapted Riemannian metric we want.
	\end{proof}
	
	  Assume that $\mu\neq |\mu^s_i(A)|={\rm exp}\big(\lambda^s_i(A)\big)$. Fix $\delta< {\rm min}  \big\{  |\frac{\mu}{\mu^s_i(A)}-1|,  |\frac{\mu^s_i(A)}{\mu}-1|  \big\}$  and an adapted norm $\|\cdot\|$ from Claim \ref{adapted metric on Rd}. Let $H:\mathbb{R}^d\to \mathbb{R}^d$ be the conjugacy defined in Proposition \ref{lifting conjugate} satisfying $|H-Id|\leq C_0$. We take two points $x,y \in \mathbb{R}^d$ such that $y\in \tildeF^s_i(x)$.   One has
		$$\mu^{-k}(1+\delta)^{-k}\cdot d^s(x,y)\leq d^s(F^{-k}x, F^{-k}y )\leq \mu^{-k}(1+\delta)^{k}\cdot d^s(x,y),$$
		further,
		\begin{align}
			a\cdot\mu^{-k}(1+\delta)^{-k}\cdot d^s(x,y)\leq d(F^{-k}x, F^{-k}y)\leq \mu^{-k}(1+\delta)^{k}\cdot d^s(x,y), \label{prop3.8.1}
		\end{align}
		where $a$ is given by $\eqref{defquasi-isometric}$, since $\tildeF^s_i$ is quasi-isometric (by Proposition \ref{leaf proposition}).
		Meanwhile, since  $H(\tildeF^s_i)=\tildeL^s_i$ ( when dim$E^s=1$, $H(\tildeF^s)=\tildeL^s$ always holds),
		\begin{align}
			d\left(H(F^{-k}x), H(F^{-k}x)\right)= d\left(A^{-k}(Hx), A^{-k}(Hx)\right)=\left(\mu^s_i(A)\right)^{-k}\cdot d\left(Hx,Hy\right).\label{prop3.8.2}
		\end{align}
		The formulas \eqref{prop3.8.1} and \eqref{prop3.8.2} jointly contradict with the fact
		$$\Big| d(F^{-k}x, F^{-k}y)-  d\left(H(F^{-k}x), H(F^{-k}x)\right)  \Big| \leq 2C_0.$$
	\end{proof}
	
	For the convenience of readers, we state Journ$\acute{\rm e}$ Lemma \cite{journe} as the following proposition which will be useful in Section \ref{section affine metric} and Section \ref{special}.
	\begin{proposition}[ \cite{journe}]\label{journe}
		Let $M_i \; (i=1,2)$ be a smooth manifold and $\mathcal{F}_i^s, \mathcal{F}_i^u$ be continuous transverse foliations on $M_i$ with uniformly $C^{r+\alpha}$-smooth leaves $(r\geq 1, 0<\alpha<1 )$. Assume that $h:M_1\to M_2$ is a homeomorphism and maps $\mathcal{F}_1^{\sigma}$ to $\mathcal{F}^{\sigma}_2 (\sigma=s,u)$. If $h$ restricted on leaves of both $\mathcal{F}_1^s$ and $\mathcal{F}_1^u$ is uniformly $C^{r+\alpha}$, then $h$ is $C^{r+\alpha}$-smooth.
	\end{proposition}

\section{Spectral rigidity on stable bundle}\label{s-rigidity}
	
	In this section, we prove  the necessary parts of both Theorem \ref{main theorem 1} and Theorem \ref{main theorem 2}. As mentioned, we can actually prove them under $C^1$ assumption. For convenience, we restate these as follow.
	
	\begin{theorem}\label{special implies s-rigidity}
		Let $A:\TT^d\to \TT^d$ be an irreducible linear non-invertible Anosov map. Assume that $A$ admits the finest (on stable bundle) dominated splitting,
		$$T\TT^d=L^s_1\oplus  L^s_2\oplus ...\oplus  L^s_k\oplus L^u,$$
		where ${\rm dim}L^s_i=1$, $1\leq i\leq k$.
		
		Then there exists a $C^1$ neighborhood $\mathcal{U}\subset C^1(\TT^d)$ of $A$ such that for every $f\in \mathcal{U}$, if $f$ is special, then
		$\lambda^s_i(p,f)=\lambda^s_i(A)$, for all $p\in {\rm Per}(f)$ and all $1\leq i\leq k$.
		Moreover, $f$ admits the finest (on stable bundle) dominated splitting,
			$$T\TT^d=E^s_1\oplus E^s_2\oplus...\oplus E^s_k\oplus E^u,$$
		where {\rm dim}$E^s_i=1$, for all $1\leq i\leq k$.
		
		Moreover, when $k=1$, for every $C^1$-smooth non-invertible Anosov map $f$ with irreducible linearization $A$, if $f$ is special, then $\lambda^s(p,f)=\lambda^s(A)$, for all $p\in{\rm Per}(f)$.
	\end{theorem}
	
	Now, we give the scheme of our proof. In this section, we always assume that  $A:\TT^d\to \TT^d$ satisfies the condition of Theorem \ref{special implies s-rigidity}.
	 To get the spectral rigidity on stable bundle , we firstly prove that every periodic point $p\in{\rm Per}(f)$ has the same \textit{stable Lyapunov spectrum} $\big\{ \lambda^s_i(p,f) : i=1,2,...,k\big\}$.
	\begin{proposition}\label{periodic data}
		Let $f:\TT^d\to \TT^d$ be an irreducible non-invertible Anosov map with a $Df$-invariant one-dimensional subbundle $E^s_i\subset E^s$. If $f$ is special and there exists an $f$-invariant foliation on $\TT^d$ tangent to $E^s_i$, then $\lambda^s_i(p,f)=\lambda^s_i(q,f)$, for all $p,q\in {\rm Per}(f)$, where $\lambda^s_i(p,f)$ is the Lyapunov exponent of $f$ for $p$ corresponding the bundle $E^s_i$.
	\end{proposition}
	
	We emphasize here that in the proof of Proposition \ref{periodic data}, $f$ need not be a small perturbation of $A$. To get that every periodic point of $f$ has the same stable Lyapunov spectrum through Proposition \ref{periodic data}, we need that $f$ admits the finest (on stable bundle) dominated splitting.

	\begin{proposition}\label{finest dominated splitting of f}
		There exists a $C^1$ neighborhood $\mathcal{U}\subset C^1(\TT^d)$ of $A$ such that for every  $f\in \mathcal{U}$, if it is special, then it admits the finest (on stable bundle) dominated splitting
		$$T\TT^d=E^s_1\oplus E^s_2\oplus...\oplus E^s_k\oplus E^u,$$
		where $E^s_i$ is one-dimensional and integrable, for all $1\leq i\leq k$.
	\end{proposition}
	
	We leave the proofs of Proposition \ref{periodic data} and Proposition \ref{finest dominated splitting of f} in subsection \ref{subsection 3.1}.

	To obtain the relationship between the periodic stable Lyapunov spectrum  and one of its linearization, we can use Proposition \ref{s-exp coincide with linear}. For a special $f\in\mathcal{U}$ given by Proposition \ref{finest dominated splitting of f},  let $h$ be the conjugacy between $f$ and $A$ given by Proposition \ref{special and conjugate}.  We need to prove that the conjugacy $h$ is also a leaf conjugacy between $\mathcal{F}^s_i$ and $\mathcal{L}^s_i$.

	\begin{proposition} \label{su-integrable and leaf-conjugate}
	 Let $f\in\mathcal{U}$ given by Proposition \ref{finest dominated splitting of f} be special. Then 	$h(\mathcal{F}^s_i)=\mathcal{L}^s_i$, for every $1\leq i\leq k$.
	\end{proposition}
	
    We leave the proof of Proposition \ref{su-integrable and leaf-conjugate} in subsection \ref{subsection 3.2}.
	Now, we can prove  Theorem \ref{special implies s-rigidity}.
	\begin{proof}[Proof of Theorem \ref{special implies s-rigidity}]
		For one-dimensional stable bundle case, the special Anosov map $f$ admits the dominated splitting $T\TT^d=E^s\oplus E^u$ and $H(\tildeF^s)=\tildeL^s$  always holds whether $f$ is a small perturbation  of its linearization or not. Hence by Proposition \ref{s-exp coincide with linear} and Proposition \ref{periodic data}, we get $\lambda^s(p,f)=\lambda^s(A)$, for all $p\in {\rm Per}(f)$, immediately.
		
		For higher-dimensional stable bundle case,  by Proposition \ref{finest dominated splitting of f}, there exists a $C^1$ neighborhood $\mathcal{U}$ of $A$ such  that every special $f\in\mathcal{U}$ admits the finest (on stable bundle) dominated splitting. Thus by Proposition \ref{periodic data}, every periodic point $p\in{\rm Per}(f)$ has the same stable Lyapunov spectrum. Now,
		combining Proposition \ref{s-exp coincide with linear} and Proposition \ref{su-integrable and leaf-conjugate}, we have that
		$\lambda^s_i(p,f)=\lambda^s_i(A)$, for every $p\in {\rm Per}(f)$ and every $1\leq i\leq k$.
	\end{proof}

	\subsection{Periodic stable Lyapunov spectrums coincide}\label{subsection 3.1}
	In this subsection, we prove Proposition \ref{periodic data} and Proposition \ref{finest dominated splitting of f}. Fix $1\leq i\leq k$, let $f:\TT^d\to \TT^d$ be a special irreducible non-invertible Anosov map with $Df$-invariant subbundle $E^s_i\subset E^s$. Let $\mathcal{F}^s_i$ be an $f$-invariant integral foliation for $E^s_i$.  For short, we denote $\mu^s_i(p,f):={\rm exp}\left(\lambda^s_i(p,f)\right)$ by $\mu^s_i(p)$, for all $p\in{\rm Per}(f)$.
	
	\begin{proof}[Proof of Proposition \ref{periodic data}]\label{proof of rigidity}
	We assume that there exist $p,q\in \rm{Per}$$(f)$ such that $\mu^s_i(p)<\mu^s_i(q)$, then to get a contradiction.
	By the assumption of the existence of different periodic stable Lyapunov exponents, the infimum $\mu_-$ and the supremum $\mu_+$ of the set $\big\{\mu_i^s(p):p\in{\rm Per}(f)\big\}$ satisfy $0<\mu_-<\mu_+<1$.
	Given $\delta>0$ arbitrarily small, we can choose two periodic points $p,q$ of $f$ such that
		$$
		\mu^s_i(p)\le\mu_- \cdot (1+\delta) \quad {\rm{and}} \quad \mu^s_i(q)\ge {\mu_+}\cdot(1+\delta)^{-1},
		$$
		Moreover, as Claim \ref{adapted metric on Rd}, there exists a smooth adapted Riemannian metric such that
		$$
		\mu_-\cdot(1+\delta)^{-1} <\|Df|_{E_i^s(x)}\|<\mu_+\cdot(1+\delta), \quad\forall x\in\mathbb{T}^d.
		$$

		For convenience, we can assume that $p,q$ are both fixed points. Otherwise we can go through the rest of this proof by using $f^{n_0}$ instead of $f$, where $n_0$ is the minimal common period of $p$ and $q$.
		Let $\eta_0>0$ small enough such that, for any $x_1,x_2\in \mathbb{T}^d$ with $d(x_1,x_2)\le \eta_0$,  we have
		$$(1+\delta)^{-1}\le \frac{\|Df|_{E^s_i(x_1)}\|}{\|Df|_{E^s_i(x_2)}\|}\le 1+\delta.$$
		
		Fix $\e>0$$(\e\ll\eta_0)$, there exists $x_{\e}$ in the $\e$-Ball $B_{\e}(q)$ and $k_{\e}=k(\e,x_{\e})>0$ such that  $f^{k_{\e}}(x_{\e})=p$. Indeed, since $f$ is special, the preimage set of $p$ for $f$ is dense  by Proposition \ref{special and conjugate} and Proposition \ref{preimage dense}.
		
		Shrinking $\e$, by the local product structure, there exist $\eta_1,\eta_2>0$ such that the local unstable leaf $\mathcal{F}^u(q,\eta_1)\subset B_{\eta_0}(q)$ intersects with the local stable leaf $\mathcal{F}^s(x_{\e},\eta_2)\subset B_{\eta_0}(q)$ at the unique point $y_{\e}$, namely, $y_{\e}=\mathcal{F}^s(x_{\e},\eta_2)\cap \mathcal{F}^u(q,\eta_1)$.
		Note that one has $d_{\mathcal{F}^u}(y_{\e},q) \le d(\e)$, where $d(\e) $ tends to $0$ as $\e$ goes to $0$.
		
		Therefore, we can choose a point $z_{\e}\in \mathcal{F}^s_i(x_{\e},\eta_2)$ such that $d_{\mathcal{F}^s_i}(x_{\e},z_{\e}) \ge \eta_2/3$.
		We denote by $I_{\e}$ the curve in $\mathcal{F}^s_i(x_{\e})$ from $x_{\e}$ to $z_{\e}$.
		Since $x_{\e}$ is a $k_{\e}$-preimage of the fixed point $p$, we can find a curve $J_{\e}$ in $\mathcal{F}^s_i(p)$ such that $$f^{k_{\e}}(J_{\e})=f^{k_{\e}}(I_{\e}).$$
		
		Let $N_{\e}$ be the maximal positive integer such that $d(f^j(w),f^j(q)) < \eta_0$, for all $w\in I_{\e}$ and $j\in [0,N_{\e}]$.
		Let $k_{\e}$ be the minimal positive integer such that $\big\{x\in \mathbb{T}^d | f^{k_{\e}}=p\big\} \cap B_{\e}(q) \ne \emptyset$.
		
		\begin{claim}\label{time-rate}
			There exist $\e_0>0$ and $C_0>0$ such that
			$$\frac{N_{\e}}{k_{\e}}\ge C_0, \quad \forall \e \le \e_0.$$
		\end{claim}
		
		We estimate the upper bound and the lower bound of $k_{\e}$ and $N_{\e}$, respectively. Note that we can get the lower bound of $N_{\e}$ by controlling the distance of $f^{N_{\e}}(y_{\e})$ and $f^{N_{\e}}(q)$ along unstable leaves directly.
		However, it is difficult to estimate the upper bound of $k_{\e}$ under the dynamics of $f$, while it is convient in linear systems (see Proposition \ref{preimage dense}). So, by Proposition \ref{special and conjugate}, let $f$ conjugate to its linearization $A$. We calculate the "$N_{\e}$" and "$k_{\e}$" of $A$. A direct way to get Claim \ref{time-rate} is using the H$\ddot{\rm o}$lder continuity of $h$, but here we prove it by only uniform continuity. See Figure 1.
		
			\begin{figure}[htbp]
			\centering
			\hspace*{-2.3cm}
			\includegraphics[width=18cm]{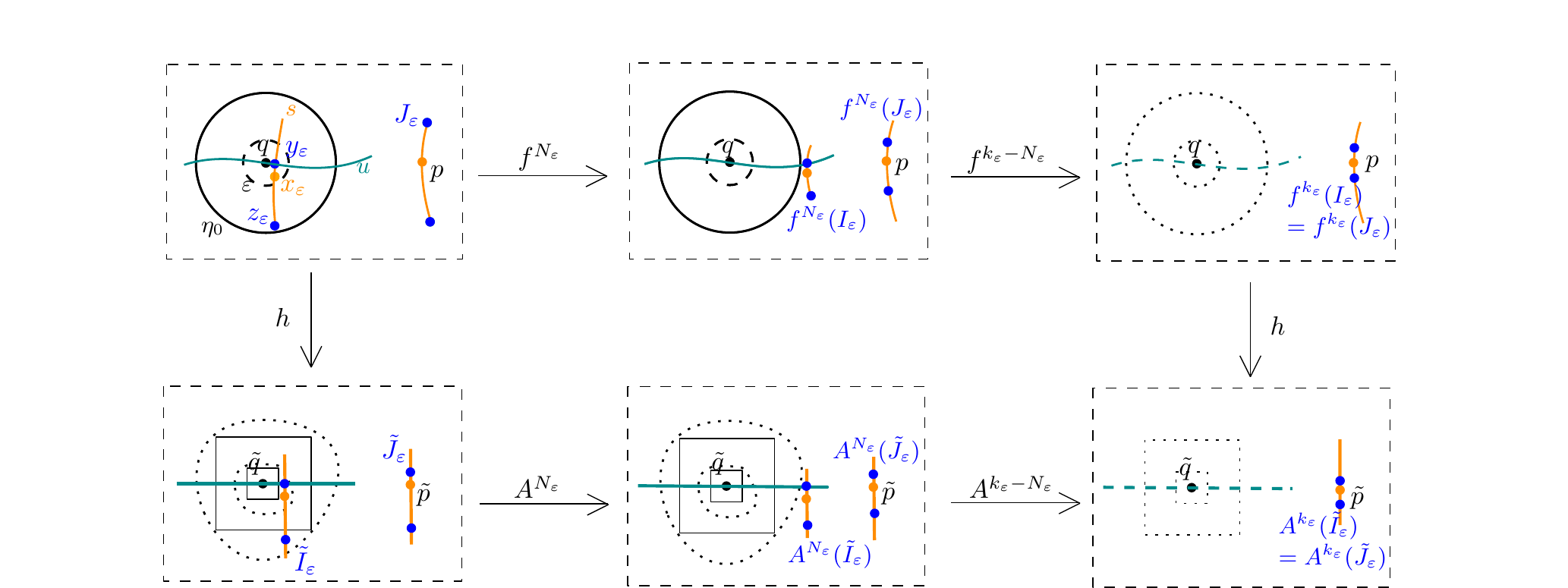}
			\caption{Iteration of local stable leaves in the case of one-dimensional stable bundle}	
		\end{figure}

		\begin{proof}[Proof of Claim\ref{time-rate}]\label{proof of time-rate}
			Let $h:\mathbb{T}^d\to \mathbb{T}^d$ be the conjugacy between $f$ and $A$ with $h\circ f= A\circ h $. For short, denote $h(x_{\e}), h(y_{\e}), h(p)$ and $h(q)$ by $\tilde{x}_{\e}, \tilde{y}_{\e}, \tilde{p}$ and $\tilde{q}$, respectively.
			The homeomorphism $h$ maps $B_{\eta_0}(q)$ and $B_{\e}(q)$ to be two neighborhoods of $\tilde{q}$ such that we can choose two su-foliation boxes of $A$, $\tilde{B}_{\tilde{\eta}_0}(\tilde{q}) \subset h(B_{\eta_0}(q))$
			and $\tilde{B}_{\tilde{\e}}(\tilde{q}) \subset h(B_{\e}(q))$.
			Here, $\tilde{\eta}_0$ is fixed by $\eta_0$, while $\tilde{\e}$ tends to $0$ following $\e$.
			Thus we can shorten $I_{\e}$ sligtly such that $\tilde{I_{\e}}:=h(I_{\e}) \subset \tilde{B}_{\tilde{\e}}(\tilde{q})\cap \mathcal{L}^s(\tilde{x}_{\e})$
			and the length $|\tilde{I_{\e}}| = \frac{\tilde{\eta}_0}{3}$.

			Let $\tilde{N_{\e}} $ be the maximal positive integer such that $d(A^j(\tilde{w}),A^j(\tilde{q})) < \tilde{\eta}_0$, for all $\tilde{w}\in \tilde{I_{\e}}$ and $j\in [0,\tilde{N_{\e}}]$.
			Let $\tilde{k}_{\e}$ be the minimal positive integer such that $\big\{x\in \mathbb{T}^d | A^{\tilde{k}_{\e}}=\tilde{p}\big\} \cap \tilde{B}_{\tilde{\e}}(\tilde{q}) \ne \emptyset$, where $\tilde{p}:=h(p)$.   	
			It is clear that $N_{\e} \ge \tilde{N_{\e}}$ and $k_{\e} \le \tilde{k}_{\e} $.
			So, we get $N_{\e}/ k_{\e} \ge \tilde{N_{\e}}/\tilde{k}_{\e}$.
			
			For every $\tilde{w}\in \tilde{I_{\e}}$ and $j\in [0,\tilde{N_{\e}}]$, we have
			\begin{align*}
				d\left(A^j(\tilde{w}),A^j(\tilde{q})\right) &\le d_{\mathcal{L}^s}\left(A^j(\tilde{w}),A^j(\tilde{y}_{\e})\right)+d_{\mathcal{L}^u}\left(A^j(\tilde{y}_{\e}),A^j(\tilde{q})\right)\\	
				&\le \frac{\tilde{\eta}_0}{3} +d(\tilde{x}_{\e}, \tilde{y}_{\e}) + \tilde{\e} \cdot  \left(\mu_{max}^u(A)\right)^j.
			\end{align*}
			Note that as $\e$ small enough, we have $d(\tilde{x}_{\e}, \tilde{y}_{\e})\leq \frac{\tilde{\eta}_0}{6}$.  So, the maximal positive integer $\tilde{N_{\e}}$ such that
			$$\frac{\tilde{\eta}_0}{3} +d(\tilde{x}_{\e}, \tilde{y}_{\e}) + \tilde{\e} \cdot  \left(\mu_{max}^u(A)\right)^j  \le\frac{\tilde{\eta}_0}{2}+\tilde{\e} \cdot  \left(\mu_{max}^u(A)\right)^j \le \tilde{\eta}_0,$$
			 holds should satisfy
			\begin{align*}
				\tilde{N_{\e}} \ge \frac{{\rm ln}\tilde{\eta}_0-{\rm ln}(2\tilde{\e})}{{\rm ln}\mu_{\rm max}^u(A)}.	
			\end{align*}
			
			On the other hand, by Proposition \ref{preimage dense}, there exists $C>0$ such that for every $\tilde{\e}>0$,
			\begin{align*}
				\tilde{k}_{\e} \le d\cdot \frac{{\rm ln}C- {\rm ln}\tilde{\e}}{{\rm ln}|{\rm det}A|}.
			\end{align*}
			Thus, as $\tilde{\e}$ tending to $0$, there exists $C_0>0$ such that,
			\begin{align*}
				\frac{N_{\e}}{k_{\e}} \ge \frac{\tilde{N_{\e}}}{\tilde{k}_{\e}} \ge C_0,
			\end{align*}
			where $C_0$ could be close to $\frac{1}{d}\cdot \frac{\rm{ln} |\rm{det}A|}{{\rm ln}\mu_{\rm max}^u(A)}$ arbitrarily.
		\end{proof}

		Now, using the uniform lower bound of the time ratio $N_{\e}/k_{\e}$ of $I_{\e}$ being around $q$ to reaching $p$, we can get an exponential error between $|f^{k_{\e}}(I_{\e})|$ and  $|f^{k_{\e}}(J_{\e})|$.
		
		Firstly, we claim that there exists $C_2>1$ such that
		\begin{align*}
			\frac{|I_{\e}|}{|J_{\e}|}\in\big[C_2^{-1},C_2\big], \quad \forall \e \ll \eta_0.	
		\end{align*}
		Indeed, by the construction of $J_{\e}$, we have $h\circ f^{k_{\e}} (J_{\e})= h\circ f^{k_{\e}} (I_{\e})$,
		equivalently, $A^{k_{\e}}(\tilde{J_{\e}})=A^{k_{\e}}(\tilde{I_{\e}})$.
		This implies that $\tilde{J_{\e}}$ is just a translation of $\tilde{I_{\e}}$,
		thus $|\tilde{J_{\e}}|=|\tilde{I_{\e}}|\ge \frac{\tilde{\eta}_0}{2}$.
		Since $h^{-1}$ is uniformly continuous, by the uniform continuity for $\mathcal{F}^s_i$ (see Remark \ref{uniform continuity of foliation on Rd}), we have that   $\frac{|I_{\e}|}{|J_{\e}|}$ is uniformly bounded away from $0$.
		
    	Note that we can assume $J_{\e} \subset B_{\eta_0}(p)$. Otherwise, by the uniform continuity of $h$ again, we can shorten the length of $I_{\e}$, meanwhile ensure that the length is independent of $\e$.
	
	    Since we have assumed that $p,q$ are fixed points of $f$, then $f^j(J_{\e}) \subset B_{\eta_0}(p)$, for every $j\ge0$.
		So,
		\begin{align*}
			\big|f^{k_{\e}}(J_{\e})\big| \le \big( \mu_- \cdot (1+\delta)^2 \big)^{k_{\e}}\big|J_{\e}\big|.
		\end{align*}
		And,
		\begin{align*}
			|f^{k_{\e}}(I_{\e})|
			&=|f^{k_{\e}-N_{\e}} \circ f^{N_{\e}} (I_{\e})|\\
			&\ge \left(\frac{\mu_-}{1+\delta}\right)^{k_{\e}-N_{\e}} \cdot \left(\frac{\mu_+}{(1+\delta)^2}\right)^{N_{\e}} |I_{\e}|.
		\end{align*}
		Consquently,
		\begin{align*}
			\frac{|f^{k_{\e}}(I_{\e})|}{|f^{k_{\e}}(J_{\e})|}
			& \ge \frac{1}{(1+\delta) ^ {3k_{\e}+N_{\e}}} \cdot \left(\frac{\mu_+}{\mu_-}\right)^{N_{\e}} \cdot \frac{|I_{\e}|}{|J_{\e}|}
			\ge \frac{1}{(1+\delta) ^ {4k_{\e}}} \cdot \left(\frac{\mu_+}{\mu_-}\right)^{N_{\e}} \cdot \frac{1}{C_2}\\
			& \ge \frac{1}{C_2 \cdot (1+\delta) ^ {4k_{\e}}} \cdot \left(\frac{\mu_+}{\mu_-}\right)^{C_0\cdot k_{\e}}
			= \frac{1}{C_2} \cdot \Big((1+\delta) ^ 4 \cdot \left(\frac{\mu_+}{\mu_-}\right)^{C_0}\Big)^ {k_{\e}}.
		\end{align*}
		We can assume that $k_{\e} \ge N_{\e}$, so that we have the second inequality. Otherwise, when $k_{\e} < N_{\e}$, the whole estimation of $\frac{|f^{k_{\e}}(I_{\e})|}{|f^{k_{\e}}(J_{\e})|}$ is trivial.

		Let $\delta$ is small enough such that $(1+\delta) ^ 4 \cdot \left(\frac{\mu_+}{\mu_-}\right)^{C_0} >1$,
		after that  we can fix $\eta_0$.
		Let $\e$ tend to zero. We have $k_{\e}\to +\infty$,
		hence $\frac{|f^{k_{\e}}(I_{\e})|}{|f^{k_{\e}}(J_{\e})|}\to +\infty$.
		This contradics the fact that $f^{k_{\e}}(I_{\e})=f^{k_{\e}}(J_{\e})$.
	\end{proof}
	
	Now, we show that there exists a $C^1$ neighbohood $\mathcal{U}$ of $A$ in which every special $f$ admits the finest (on stable bundle) dominated splitting. Combining with Proposition \ref{periodic data}, if $f\in\mathcal{U}$ is special, then $\lambda^s_i(p,f)= \lambda^s_i(q,f)$, for all $p,q\in{\rm Per}(f)$ and $1\leq i\leq k$.
	
	\begin{proof}[Proof of Proposition \ref{finest dominated splitting of f}]
			Since $E^s_i=E^s_{(1,i)}\cap E^s_{(i,k)}$, it suffices to prove the following lemma.
		\begin{lemma}\label{special and simplest dominated}
		There exists a $C^1$ neighborhood $\mathcal{U}\subset C^1(\TT^d)$ of $A$ such that, for every $f\in \mathcal{U}$ and  $1\leq i\leq k-1$, if $f$ is special, then it admits the following dominated splitting
		$$E^s_{(1,i)}\oplus E^s_{(i+1,k)}\oplus E^u,$$
		and $E^s_{(i+1,k)}$ is integrable.
	\end{lemma}
	\begin{proof}[Proof of Lemma \ref{special and simplest dominated}]
		By Proposition \ref{s-bundle-continuous}, $E^s_{(1,i)}$ is well defined on $\TT^d$. By the assumption that $f$ is special, the unstable bundle $E^u$ is also well defined on $\TT^d$.
		
		Let $F$ be a lifting of $f$ and $H$ be the conjugacy between $F$ and $A$.	As the proof of Proposition \ref{leaf proposition}, there exists a $C^1$ neighborhood $\mathcal{U}\subset C^1(\TT^d)$ of $A$ such that for every $f\in\mathcal{U}$, its lifting $F$ admits a dominated splitting
		$$\tilde{E}^s_{(1,i)}\oplus_<\tilde{E}^s_{(i+1,k)}\oplus_<\tilde{E}^u,$$
		and $\tilde{E}^s_{(i+1,k)}$ is integrable.  Moreover, by the forth item of Proposition \ref{leaf proposition}, $H(y)\in \tildeL^s_{(i+1,k)}(H(x))$  if and only if $y\in \tildeF^s_{(i+1,k)}(x)$. Note that, since $f$ is special, $H(x+n)=H(x)+n$, for every $x\in\RR^d$ and $n\in\ZZ^d$.
		Thus, we have that
		\begin{align*}
			y+n\in \tildeF^s_{(i+1,k)}(x+n)
			& \Longleftrightarrow H(y+n)\in \tildeL^s_{(i+1,k)}\big(H(x+n)\big),\\
			& \Longleftrightarrow	H(y)+n\in \tildeL^s_{(i+1,k)}\big(H(x)+n\big),\\
			& \Longleftrightarrow	H(y)\in \tildeL^s_{(i+1,k)}\big(H(x)\big),\\
			& \Longleftrightarrow y\in \tildeF^s_{(i+1,k)}(x).
		\end{align*}
		It means that $\tilde{E}^s_{(i+1,k)}$ is  $\ZZ^d$-periodic, hence it can descend to $\TT^d$ through \eqref{bundleonorbitandlift}.
	\end{proof}
	\end{proof}

	\subsection{The conjugacy preserves strong stable foliations}\label{subsection 3.2}
	Now we prove Proposition \ref{su-integrable and leaf-conjugate} that is $h(\mathcal{F}^s_i)=\mathcal{L}^s_i$, for all $1\leq i\leq k$, where $h$ is the conjugacy between the special $f$ and its linearization $A$.
	
	\begin{proof}[Proof of Proposition \ref{su-integrable and leaf-conjugate}]
		  By Proposition \ref{leaf proposition}, we already have $h\left(\mathcal{F}^s_{(i,k)}\right)= \mathcal{L}^s_{(i,k)}$, for every $1\leq i\leq k$. Since
		$\mathcal{F}^s_i= \mathcal{F}^s_{(1,i)}\cap\mathcal{F}^s_{(i,k)}$, it suffices to prove the following lemma.
		
	\begin{lemma}\label{lemma su-integrable and leaf-conjugate}
		Let $f\in\mathcal{U}$ given by Proposition \ref{finest dominated splitting of f} be special. Then for every $1\leq i\leq k$,
		$h(\mathcal{F}^s_{(1,i)})=\mathcal{L}^s_{(1,i)}$.
	\end{lemma}

	\begin{proof}[Proof of Lemma \ref{lemma su-integrable and leaf-conjugate}]
		Fix $1\leq i\leq k-1$. Firstly, we prove the joint integrability of the bundle $E^s_{(1,i)}\oplus E^u$.
	
		\begin{claim}\label{su-jointlyclaim}
			The bundle	$E^{s,u}_{(1,i)}=E^s_{(1,i)}\oplus E^u$ is jointly integrable.
		\end{claim}
		\begin{proof}[Proof of Claim \ref{su-jointlyclaim}]
			For any $x\in\TT^d$,  $y\in \mathcal{F}^s_{(1,i)}(x,\delta)\setminus \{x\}$ and $x'\in \mathcal{F}^u(x)\setminus \{x\}$, let $y'={\rm Hol}^u_{x,x'}(y)\in\mathcal{F}^s(x')$, where Hol$^u_{x,x'}: \mathcal{F}^s(x,\delta)\to \mathcal{F}^s(x',\delta)$ is the holonomy map along the unstable foliation $\mathcal{F}^u$. It suffices to show that  $y'\in\mathcal{F}^s_{(1,i)}(x')$.
			
			Let $I$ be any curve homeomorphic to $[0,1]$ and laying on $\mathcal{F}^s_{(1,i)}(x,\delta)$ with endpoints $x,y$. Let $J={\rm Hol}^u_{x,x'}(I)$. Since the conjugacy $h$ maps $\mathcal{F}^u$ to $\mathcal{L}^u$, the curve $h(J)$ is a translation of $h(I)$. Now, by Proposition \ref{preimage dense}, there exists $z_n\to h(x')$ with $A^nz_n=A^nh(x)$. Moreover, we can pick curves $I_n$ with $z_n\in I_n$ such that $A^nI_n=A^nh(I)$. Hence, the other endpoint $w_n(\neq z_n)$ of $I_n$ also has $A^nw_n=A^nh(y)$. Since $I_n$ is also a translation of $h(I)$, one has $I_n\to h(J)$ and $w_n\to h(y')$. See Figure 2.
			
				\begin{figure}[htbp]
				\centering
				\includegraphics[width=16cm]{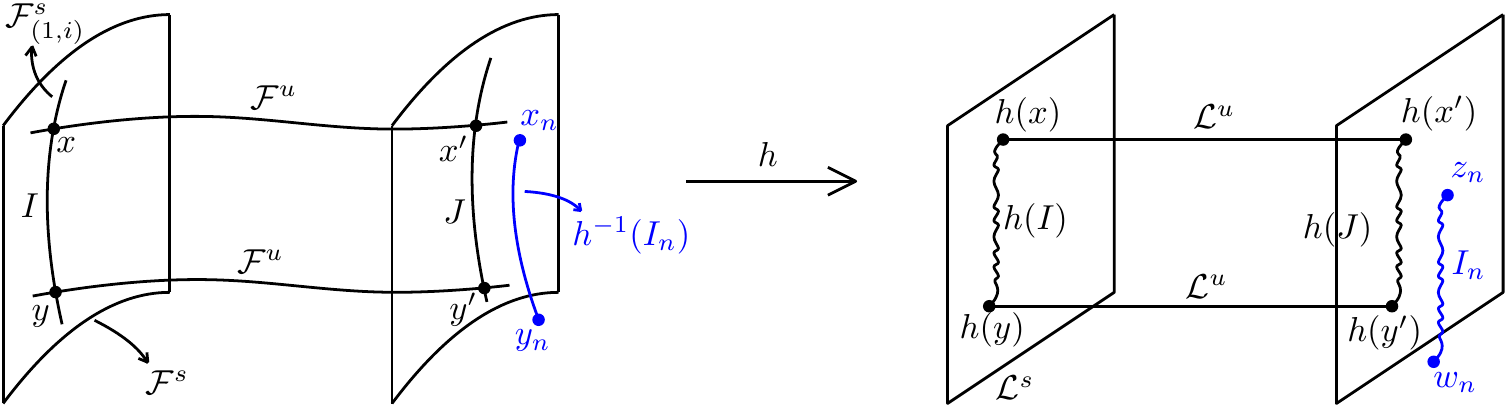}
				\caption{The image of holonomy map is approached by preimage sets.}	
			\end{figure}

			Let $x_n=h^{-1}(z_n)$ and $y_n=h^{-1}(w_n)$.  We have that $x_n\to x'$ and $y_n\to y'$.  Since $A^nI_n=A^nh(I)$, one has $f^n(h^{-1}I_n)=f^n(I)$. It follows that $h^{-1}(I_n)$ is a curve laying on $\mathcal{F}^s_{(1,i)}(x_n)$ with endpoints $x_n$ and $y_n$. Moreover, by the continuity of $h$ and $h^{-1}$, the curve $h^{-1}(I_n)$ is  located in a small tubular neighborhood of $J$, when $n$ is big enough. Hence, we have that $h^{-1}(I_n)\to J$. By the continuity of the foliation $\mathcal{F}^s_{(1,i)}$, we get $J\subset \mathcal{F}^s_{(1,i)}(x')$ and $y'\in \mathcal{F}^s_{(1,i)}(x')$.
		\end{proof}
		
		Now, by the fact that $\mathcal{F}^{s,u}_{(1,i)}$ is  subfoliated by the unstable foliation $\mathcal{F}^u$  which is minimal, we can get that  $h(\mathcal{F}^s_{(1,i)})$ is a linear foliation.
		\begin{claim}\label{linearfoliationclaim}
			$h(\mathcal{F}^s_{(1,i)})$ is a linear foliation.
		\end{claim}
		\begin{proof}[Proof of Claim \ref{linearfoliationclaim}]
			For convenience, we prove it on the universal cover space $\RR^d$.
			Let $\tilde{\mathcal{W}}=H(\tildeF^{s,u}_{(1,i)})$. Note that
			$$ H(\tildeF^s)=\tildeL^s,\quad H(\tildeF^u)=\tildeL^u, \quad H(\tildeF^s_{(1,i)})=\tilde{\mathcal{W}}\cap \tildeL^s.$$
			We are going to prove that $\tilde{\mathcal{W}}$ is additively closed, that is $x+y\in \tilde{\mathcal{W}}(0),$
			for all $x,y\in\tilde{\mathcal{W}}(0)$. Combining with the fact that $\tildeL^s$ is additively closed, we have that $H(\tildeF^s_{(1,i)})=\tildeL^s\cap \tilde{\mathcal{W}}$ is additively closed. Hence it is a linear foliation.
			Note that
			\begin{enumerate}
				\item $\tilde{\mathcal{W}}(x+n)=\tilde{\mathcal{W}}(x)+n$, for all $x\in \RR^d$ and $n\in \ZZ^d$.
				\item $\tilde{\mathcal{W}}(x+v)=\tilde{\mathcal{W}}(x)$, for all $x\in \RR^d$ and $v\in \tildeL^u(0)$.
			\end{enumerate}
			Indeed,  we just need the fact that $\tildeL^u$ is linear and
			$\tildeF^u(x+n)=\tildeF^u(x)+n$, for all $x\in \RR^d$ and $n\in \ZZ^d$.
			
			Since the foliation $\mathcal{L}^u$ is minimal, there exist $n\in \ZZ^d$ and $v_n\in \tildeL^u(0)$ such that $n+v_n\to x$. By $y\in \tilde{\mathcal{W}}(0)$, we have $v_n+y\in \tilde{\mathcal{W}}(0)$.  Hence, $n+v_n+y\in\tilde{\mathcal{W}}(n)=\tilde{\mathcal{W}}(n+v_n)$. Let $n+v_n\to x$, it follows that $x+y\in \tilde{\mathcal{W}}(x)=\tilde{\mathcal{W}}(0)$.
		\end{proof}
		
		Now, we can finish our proof of Lemma \ref{lemma su-integrable and leaf-conjugate}.
		Note that the linear foliation $H(\tildeF^s_{(1,i)})\subset \tildeL^s$ is $A$-invariant. So, it must be a union foliation of $i$ subfoliations of $\tildeL^s$, that is $H(\tildeF^s_{(1,i)})=\tildeL^s_{j_1}\oplus...\oplus \tildeL^s_{j_i}.$
		On the other hand, by Proposition \ref{leaf proposition}, we have
		$ H(\tildeF^s_{(i+1,k)})=\tildeL^s_{(i+1,k)}.$
		Since $H$ is a homeomophism, we get $H(\tildeF^s_{(1,i)})=\tildeL^s_{(1,i)}$.
	\end{proof}
	\end{proof}

	\section{Affine metric}\label{section affine metric}
	Let $f$ belong to $\mathcal{U}$ given by Proposition \ref{leaf proposition}. Applying Livschitz Theorem  for Anosov maps (Proposition \ref{Livsic}), the spectral rigidity on the stable bundle implies that we can endow with a metric to each stable foliation $\tildeF^s_i$ such that $F$ is affine restricted on it.  Especially, for the case of dim$E^s=1$, the existence of the affine metric is just like the diffeomorphism case (\cite[Lemma 3.1]{ganshirigidity}) whether $f$ is a small perturbation or not. However, for the case of dim$E^s>1$, there are something quite different: we do not \textit{a priori} have foliation $\mathcal{F}^s_i$  on $\mathbb{T}^d$ for non-invertible Anosov maps. So, we will give this affine metric on the lifting $\tildeF^s_i$.   Moreover, lack of the bundle $E^s_i (2\leq i \leq k)$ on $T\TT^d$ prevents us from defining the contracting-rate function for $E^s_i$.  We overcome this by using quotient dynamics.
	
	\begin{proposition}[Livschitz Theorem]\label{Livsic}
		Let $M$ be a closed Riemannian manifold, $f:M\to M$ be a $C^{1+\alpha}$ transitive Anosov map and $\phi: M \to \mathbb{R} $ be a H$\ddot{o}$lder continuous function.
		Suppose that, for every $p\in \rm{Per}$$(f)$ with period $\pi(p)$, $\sum_{i=0}^{\pi(p)-1}\phi(f^i(p))=0$ .
		Then there exists a continuous function $\psi: M \to \mathbb{R}$ such that $\phi=\psi\circ f - \psi$. Moreover $\psi$ is unique up to an additive constant and H$\ddot{o}$lder with the same exponent as $\phi$.
	\end{proposition}
	Note that the proof of this  proposition is quite similar to the Livschitz Theorem for transitive Anosov diffeomorphisms whose complete proof can be found in \cite[Corollary 6.4.17 and Theorem 19.2.1]{katok}.
	If one need, we also refer to \cite[ Theorem 2.1 and Proposition 2.3 ]{Przytycki} for the (un)stable manifolds theorem and the local product structure for Anosov maps, which are useful to prove the Anosov (exponential) closing lemma (see \cite[Corollary 6.4.17 ]{katok}) and the Livschitz Theorem for Anosov maps.
	
	We also mention that every Anosov map on torus is transitive, so that the Livschitz Theorem always holds for toral Anosov maps.
	If one need, we refer to \cite[Theorem 6.8.1]{aoki} which says that $\sigma_f:\mathbb{T}^d_f\to\mathbb{T}^d_f$ the inverse limit system of $f$ is topologically conjugate to $\sigma_A:\mathbb{T}^d_A\to\mathbb{T}^d_A$ the inverse limit system of $A$. Combining the fact that the toral Anosov map $A$ is transitive, we get the transitivity of $f$. We also refer to \cite[Proposition 1.2 ]{robustspecial} to get a proof for transitivity of Anosov maps on infra-nilmanifolds.
	
	\vspace{0.3cm}

	Now, using quotient dynamics, we can apply Livschitz Theorem to $F:\RR^d\to \RR^d$.
	\begin{proposition}\label{quotient bundle}
		Let $f\in\mathcal{U}$ given by Proposition \ref{leaf proposition}. Fix $1\leq i\leq k$, if $\lambda^s_i(p,f)=\lambda^s_i(A)$, for every $p\in{\rm Per}(f)$, then there exists a  H$\ddot{o}$lder continuous function $\Psi:\RR^d\to \RR$, such that
		$${\rm log}\|DF|_{\tilde{E}^s_i(x)}\|=\lambda^s_i(A)+\Psi(x)-\Psi(F(x)),$$
		for every $x\in\RR^d$. Moreover,  $\Psi$ is bounded on $\RR^d$.
	\end{proposition}

	\begin{proof}
		
		For $i=1$, since $E^s_1$ is well defined on $T\TT^d$ (see Proposition \ref{s-bundle-continuous}), we can use  Livschitz Theorem (Proposition \ref{Livsic}) for function $\|Df|_{E^s_1(x)}\|$. By the assumption that  $\lambda^s_1(p,f)=\lambda^s_1(A)$ for all $p\in {\rm Per}(f)$, there exists a H$\ddot{\rm o}$lder continuous function $\psi:\TT^d\to \RR$ such that ${\rm log}\|Df|_{E^s_1(x)}\|=\lambda^s_1(A)+\psi(x)-\psi(f(x))$ for all $x\in\TT^d$.  The lifting of $\psi$, $\Psi:\RR^d\to \RR$ is the function we need.
		
		Fix $2\leq i\leq k$. Since the strong stable bundle $E^s_{(1,j)}$ $(1\le j\le k)$ is always well defined  on $T\TT^d$, we can define $N\subset T\TT^d$ the normal bundle of $E^s_{(1,i-1)}$ in $E^s_{(1,i)}$.  Let  $\pi^{N} : E^s_{(1,i)}\to N$ be the natural projection and  $\overline{D}f:E^s_{(1,i)}\to N$ defined as $\overline{D}f(x,v)=\pi^{N}\circ D_xf(v)$, for all $x\in\TT^d$ and $v\in E^s_{(1,i)}(x)$.
		
		Let $\mu(x):= \|\overline{D}f|_{N(x)}\|$ and $E^s_i(p)$ be certained by the return map of the periodic point $p$ with period $\pi(p)$, that is the eigenspace of the eigenvalue ${\rm exp}(\lambda^s_i(p,f))=\mu^s_i(A)$ for $D_pf^{\pi(p)}$. For any unit vector $v\in N(p)$,  let $v=\sum_{j=1}^{i}v^s_j$, where $v^s_j\in E^s_j(p)$.	One has
		\begin{align}
		D_pf^{\pi(p)}(v)=\sum_{j=1}^{i} (\mu^s_j(A))^{\pi(p)}v^s_j. \label{prop4.2.1}
		\end{align}
		Note that \begin{align}
			\overline{D}f(x,w)=\overline{D}f\circ\pi^{N}(x,w), \quad \forall w\in E^s_{(1,i)} (x).\label{prop4.2.2}
		\end{align}
		Indeed, let $ w\in E^s_{(1,i)}(x)$ and $w=w^s_{(1,i-1)}+w^N$ be the decomposition in $E^s_{(1,i-1)}\oplus N$, one has
		\begin{align*}
				\pi_{fx}^{N}\circ D_xf(w) & =\pi_{fx}^{N}\circ D_xf\left(w^s_{(1,i-1)}+w^N\right)=\pi_{fx}^N\left (   D_xfw^s_{(1,i-1)}+ D_xfw^N\right ), \\
				& =\pi_{fx}^N\circ D_xfw^s_{(1,i-1)}+\pi_{fx}^N\circ D_xfw^{N}=\pi_{fx}^N\circ D_xf(w^{N}).
		\end{align*}
	  Hence, by \eqref{prop4.2.1} and \eqref{prop4.2.2},
			$$(\overline{D}f)^{\pi(p)}(x,v)=\pi^N_p \circ D_pf^{\pi(p)}(v)=\left(\mu^s_i(A)\right)^{\pi(p)}\pi^N_p( v^s_i)
			=\left(\mu^s_i(A)\right)^{\pi(p)}v.$$
	 It follows that $\sum_{i=0}^{\pi(p)-1}{\rm log}\mu(f^i(p))=\pi(p) \cdot \mu^s_i(A)$ for all $p\in{\rm Per}(f)$.
      Now, using Livschitz Theorem for log$\mu(x)$,  there exists a H$\ddot{\rm o}$lder continuous  function $\phi: \TT^d \to \RR$ such that
			$${\rm log}\mu(x)=\lambda^s_i(A)+\phi(x)-\phi(f(x)).$$

		Let $ \overline{D}F:\tilde{E}^s_{(1,i)}\to \tilde{N}$ be the lifting of $ \overline{D}f:E^s_{(1,i)}\to N$, where $\tilde{N}\subset T\RR^d$ is the lifting of $N$. Let  $\Phi:\RR^d\to \RR$ be the lifting of $\phi:\TT^d\to \RR$. Denote $ \tilde\mu(x):= \|\overline{D}F|_{\tilde{N}(x)}\|$. Thus, we have
		\begin{align}
			{\rm log}\tilde\mu(x)=\lambda^s_i(A)+\Phi(x)-\Phi(F(x)).  \label{quotient.1}
		\end{align}
		Note that $\Phi$ is bounded  and H$\ddot{\rm o}$lder continuous on $\RR^d$.
		
			Let $\alpha(x):= {\rm log\;cos} \angle(\tilde{N}(x), \tilde{E}^s_i(x))$,  we claim that
			\begin{align}
				{\rm log}\|DF|_{\tilde{E}^s_i(x)}\|={\rm log}\tilde\mu(x)+\alpha(x)-\alpha(F(x)).  \label{quotient.2}
			\end{align}
			Indeed, it is just a linear algebraic calculation. Let $v\in \tilde{E}^s_i(x)$, we have
			$$\|v^{\tilde{N}}\|=\text{cos}  \angle(\tilde{N}(x), \tilde{E}^s_i(x))\cdot\|v\|= e^{\alpha(x)}\cdot \|v\|,$$ and by \eqref{prop4.2.2} (this equation can also lift on $T\RR^d$),
			$$ \tilde\mu(x)\|v^{\tilde{N}}\|=\pi_{Fx}^{\tilde{N}} \circ D_xF(v)= \text{cos} \angle(\tilde{N}(Fx), \tilde{E}^s_i(Fx))\cdot\|D_xFv\|= e^{\alpha(Fx)}\cdot \|D_xFv\|.$$ Thus, we get \eqref{quotient.2}.

		Now,  by \eqref{quotient.1} and \eqref{quotient.2}, $\Psi=\Phi+\alpha$ is a function satisfying
		$${\rm log}\|DF|_{\tilde{E}^s_i(x)}\|=\lambda^s_i(A)+\Psi(x)-\Psi(F(x)).$$
		Moreover, since the angle $ \angle(\tilde{N}(x), \tilde{E}^s_i(x))$ is uniformly away from $\pi/2$ (see Proposition \ref{s-bundle-continuous} and \eqref{epsilon0}), $\alpha(x)$ is a bounded function defined on $\RR^d$. And, so is $\Psi$. The  H$\ddot{\rm o}$lder continuity of both $ \angle(\tilde{N}(x), \tilde{E}^s_i(x))$ (see Remark \ref{holderbundle}) and function $\Phi$ implies one of  $\Psi$.
	\end{proof}

	Now, we can endow an affine metric to each leaf of $\tildeF^s_i$ and it would be invariant under some certain holonomy maps. Let foliation $\tildeF$ be subfoliated by foliations $\tildeF_1$ and $\tildeF_2$ which admit the Global Product Structure on $\tildeF$. We define the \textit{holonomy map} of $\tildeF_1$ along $\tildeF_2$ restricted on $\tildeF$ as
	\begin{align*}
		 {\rm Hol}_{x,x'}:\tildeF_1(x)\to \tildeF_1(x') \quad {\rm with } \quad
		 {\rm Hol}_{x,x'}(y)=\tildeF_1(x')\cap \tildeF_2(y),
	\end{align*}
  for every $x,x'\in\RR^d and y\in\tildeF_1(x)$.  Let $d^s_i(\cdot,\cdot)$ be a metric defined on  each leaf of  $\tildeF^s_i$, we say $d^s_i(\cdot,\cdot)$ is \textit{continuous}, if  for any $\e>0$, $x\in\RR^d$ and $y\in\tildeF^s_i(x)$, there exists $\delta>0$ such that $$\left|  d^s_i(x,y)-d^s_i(x',y')       \right|<\e,$$ for all $x'\in B(x,\delta)$ and $y'\in B(y,\delta)$ with $y'\in\tildeF^s_i(x')$.	

	\begin{proposition}\label{affine metric}
			Let $f\in\mathcal{U}$ given by Proposition \ref{leaf proposition}. Fix $1\le i\le k$, if $\lambda^s_i(p,f)=\lambda^s_i(A)$, for every $p\in{\rm Per}(f)$, then there exists a  continuous metric  $d^s_i(\cdot,\cdot)$ defined on  each leaf of  $\tildeF^s_i$ satisfying,
		\begin{enumerate}
			\item There exists a constant $K>1$, such that $1/K\cdot d_{\tildeF^s_i}(x,y) < d^s_i(x,y)<K\cdot d_{\tildeF^s_i}(x,y) $, for every  $x\in \RR^d$ and $y\in \tildeF^s_i(x)$.
			\item $d^s_i(Fx,Fy)={\rm exp}(\lambda^s_i(A))\cdot d^s_i(x,y)$, for every  $x\in \RR^d$ and $y\in \tildeF^s_i(x)$.
			\item The holonomy maps of $\tildeF^s_i$ along $\tildeF^s_{(1,i-1)}(2\leq i \leq k)$ restricted on $\tildeF^s_{(1,i)}$	are isometric under the metric $d^s_i(\cdot,\cdot)$.
			\item If $\tilde{E}^s_i\oplus \tilde{E}^u$ is integrable, then the holonomy maps of $\tildeF^s_i$ along $\tildeF^u$ restricted on $\tildeF^s_i\oplus \tildeF^u$ are isometric under the metric $d^s_i(\cdot,\cdot)$.
		\end{enumerate}
	Especially, when dim$E^s=1$, if  $\lambda^s(p,f)=\lambda^s(A)$, for every $p\in{\rm Per}(f)$, there exists a continuous metric $d^s(\cdot,\cdot)$ defined on each leaf of  $\tildeF^s$ satisfying the first two items.  Moreover, the holonomy maps of $\tildeF^s$ along $\tildeF^u$ are isometric under the metric $d^s(\cdot,\cdot)$.
	\end{proposition}

	\begin{proof}[Proof of Proposition \ref{affine metric}]
		
		For every $x\in\RR^d$  and $y\in\tildeF(x)$, let $\gamma:[0,1]\to \tildeF^s_i(x)$ be a $C^1$-parametrization with $\gamma(0)=x$ and $\gamma(1)=y$. Using the same notations in Proposition \ref{quotient bundle},
		the following formula $$d^s_i(x,y):=\int_{0}^{1}e^{\Psi\circ \gamma(t)}\cdot |\gamma'(t)|dt,$$
		defines the metric we need. It is clear that the metric $d^s_i(\cdot,\cdot)$ is continuous.
		
		 Since $\Psi$ is a bounded function defined on $\RR^d$, say $\|\Psi\|_{C_0}\leq {\rm log}K$, the metric $d^s_i(\cdot,\cdot)$ is $K$-equivalent to $d_{\tildeF^s_i}(\cdot,\cdot)$. So, we get the first item.
		For the second one,  we calculate directly. Let $y\in\tildeF^s_i(x)$,
		\begin{align*}
			d^s_i(F(x),F(y))&=\int_{0}^{1}e^{\Psi\circ F\circ \gamma(t)}|(F\circ \gamma)'(t)|dt,\\
			&=\int_{0}^{1}\frac{\mu^s_i(A)}{\|DF|_{\tilde{E}^s_i(\gamma(t))}\|}\cdot e^{\Psi\circ \gamma(t)}\cdot \|DF|_{\tilde{E}^s_i(\gamma(t))}\| \cdot |\gamma'(t)|dt,\\
			&=\mu^s_i(A)d^s_i(x,y).
		\end{align*}

		Denote the holonomy maps of $\tildeF^s_i$ along $\tildeF^s_{(1,i-1)}$ restricted on $\tildeF^s_{(1,i)}$ and the holonomy maps of $\tildeF^s_i$ along  $\tildeF^u$ restricted on $\tildeF^{s,u}_i$ by Hol$^s$ and  Hol$^u$, respectively.
		
		\begin{claim}\label{uniformcontinuousmetric}
			For any $\varepsilon>0$, there exists $\delta>0$, such that, for any  two curves on $\tildeF^s_i$,
			$\gamma_1:[0,1]\to\tildeF^s_i(x)$ and $\gamma_2:[0,1]\to \tildeF^s_i(x')$. Then, each one of the follwing conditions,
			\begin{enumerate}
				\item $\gamma_2(t)={\rm Hol}_{x,x'}^s\circ \gamma_1(t)$, and $d_{\tildeF^s_{(1,i-1)}}\big(\gamma_1(t), \gamma_2(t)\big)<\delta, \forall t\in[0,1]$.
				\item  $\gamma_2(t)={\rm Hol}_{x,x'}^u\circ \gamma_1(t)$, and $d_{\tildeF^u}\big(\gamma_1(t), \gamma_2(t)\big)<\delta, \forall t\in[0,1]$.
			\end{enumerate}
			implies,
			$$\frac{d^s_i\big(\gamma_1(0),\gamma_1(1)\big)}{d^s_i\big(\gamma_2(0),\gamma_2(1)\big)} \in (1-\varepsilon,1+\varepsilon).$$
		\end{claim}
		
		\begin{proof}[Proof of Claim \ref{uniformcontinuousmetric}]
			Firstly, by the uniform continuity of $\Psi$ (see Proposition \ref{quotient bundle}), for any $\varepsilon_0>0$, there exists $\delta>0$ such that  $\|\Psi\circ \gamma_1- \Psi\circ \gamma_2\|_{C_0}\leq \varepsilon_0$.
			
			Then, we need control the deviation of holonomy maps. Applying \cite[Theorem B ]{holderfoliation}  for dominated splitting
			$$\tilde{E}^s_{(1,i-1)}\oplus \tilde{E}^s_i\oplus\tilde{E}^{s,u}_{(i+1,k)},$$
			we get that Hol$_{x,x'}^s$  is  $C^1$-smooth. Hence, we can assume there exists $\delta>0$ such that
			$|D{\rm Hol}_{x,x'}^s(\gamma_1(t))|\in(1-\varepsilon_0,1+\varepsilon_0), \forall t\in[0,1]$.
			
			Finally, we compute by definition,
			\begin{align}
				\frac{d^s_i(\gamma_1(0),\gamma_1(1))}{d^s_i(\gamma_2(0),\gamma_2(1))}&=
				\frac{\int_{0}^{1}e^{\Psi\circ \gamma_1(t)}\cdot |\gamma_1'(t)|dt}{\int_{0}^{1}e^{\Psi\circ {\rm Hol}_{x,x'}^s\circ\gamma_1(t)}\cdot |D{\rm Hol}_{x,x'}^s(\gamma_1(t))|\cdot |\gamma_1'(t)|dt},\\
				&\in\big[(1+\varepsilon_0)^{-1}e^{-\varepsilon_0}\;,\; (1-\varepsilon_0)^{-1}e^{\varepsilon_0}\big]. \label{holonomycontrol}
			\end{align}
			
			For the second case, note that by  Journ$\acute{\rm e}$ Lemma \cite{journe}(see Proposition \ref{journe}), if $\tilde{E}^s_i\oplus \tilde{E}^u$ is integrable,  then each leaf of  $\tildeF^{s,u}_i$ is $C^{1+\alpha}$. Therefore, we can use the same way of   \cite[Theorem 7.1]{pesinbook} to get  the absolute continuity of Hol$_{x,x'}^u$. In fact, we have the Radon-Nikodym derivative of ${\rm Hol}_{x,x'}^u$ as follow,
			$${\rm Jac}({\rm Hol}_{x,x'}^u)(x)= \prod_{n=0}^{\infty}\frac{\| DF|_{\tilde{E}^s_i(F^{-n}z)}  \|}{\| DF|_{\tilde{E}^s_i(F^{-n}x)}  \|},$$
			where $z={\rm Hol}_{x,x'}^u(x)$. Since the distribution $\tilde{E}^s_i$ is H$\ddot{\rm o}$lder continuous (Remark \ref{holderbundle}), by the standard distortion control techniques,  there exists $\delta>0$ such that $|{\rm Jac}({\rm Hol}_{x,x'}^u)(\gamma_1(t))|\in(1-\varepsilon_0,1+\varepsilon_0), \forall t\in[0,1]$.
			The rest of proof is similar to \eqref{holonomycontrol}, since $\tildeF^s_i$ is one-dimensional. Note that the case for Hol$^s$ can also be proved by using only absolute continuity.
		\end{proof}

		Now, we prove that Hol$_{x,x'}^u$ is isometric under the metric $d^s_i(\cdot,\cdot)$ by iterating backward. An analogical way can prove one for Hol$_{x,x'}^s$ by iterating forward.
		
		If there exist $y\in \tildeF^s_i(x)$ and $y'\in \tildeF^s_i(x')$ with Hol$_{x,x'}^u(x)=x'$ and Hol$_{x,x'}^u(y)=y'$ such that $d^s_i(x,y)\neq d^s_i(x',y')$.  Iterating these points backward, we can assume that  $\gamma_1(0)=F^{-n}(x), \gamma_1(1)=F^{-n}(y)$ and  $\gamma_2(0)=F^{-n}(x'), \gamma_2(1)=F^{-n}(y')$  satisfy the conditions of Claim \ref{uniformcontinuousmetric}, for a large $n\in\NN$. Since $F$ is affine along $\tildeF^s_i$ under the metric $d^s_i(\cdot,\cdot)$, one has
		\begin{align}
			\frac{d^s_i(F^{-n}x,F^{-n}y)}{d^s_i(F^{-n}x',F^{-n}y')}= \frac{d^s_i(x,y)}{d^s_i(x',y')}. \label{prop4.3.1}
		\end{align}
		 If we pick $\e$ small enough in Claim \ref{uniformcontinuousmetric}, then it contradicts with \eqref{prop4.3.1}. This complete the  proof.
	\end{proof}

	\section{Existence of integrable subbundles}\label{special}
	
		In this section, we prove the sufficient parts of Theorem \ref{main theorem 1} and Theorem \ref{main theorem 2} under the assumption that every periodic point of $f$ has the same Lyapunov spectrum on the stable bundle ( \textit{a priori}, need not equal to one of the linearization) and we restate as follow.
	\begin{theorem}\label{s-rigidity implies special}
		Let $A:\TT^d\to \TT^d$ be an irreducible linear Anosov map. Assume that $A$ admits the finest (on stable bundle) dominated splitting,
		$$T\TT^d=L^s_1\oplus L^s_2\oplus...\oplus L^s_k\oplus L^u,$$
		where ${\rm dim}L^s_i=1$, $1\leq i\leq k$.
		
		There exists a $C^1$ neighborhood $\mathcal{U}\subset C^1(\TT^d)$ of $A$ such that  for every $C^{1+\alpha}$-smooth $f\in \mathcal{U}$, if
		$\lambda^s_i(p,f)=\lambda^s_i(q,f)$, for all $p,q\in {\rm Per}(f)$ and all $1\leq i\leq k$, then $f$ admits the finest (on stable bundle) dominated splitting,
		$$T\TT^d=E^s_1\oplus E^s_2\oplus...\oplus E^s_k\oplus E^u,$$
		where ${\rm dim}E^s_i=1$, $1\leq i\leq k$. Especially, $f$ is special.
		
		Moreover, when $k=1$, for every $C^{1+\alpha}$ Anosov map $f$ with linearization $A$, if $\lambda^s(p,f)=\lambda^s(q,f)$,  for every $p,q\in {\rm Per}(f)$, then  $f$ is special.
	\end{theorem}
	
   We mention that in Theorem \ref{s-rigidity implies special}, $f$ can be inverse since an Anosov diffeomorphism is always a special Anosov map.

	It is convenient to give the scheme of our proof. In this section we always assume that $A$ is irreducible and has the finest (on stable bundle) dominated splitting.  Let $f\in\mathcal{U}$ given by Proposition \ref{leaf proposition} and $F:\RR^d\to\RR^d$ be a lifting of $f$ and $H:\RR^d\to\RR^d$ be the conjugacy between $F$ and $A$.
	
	Firstly, we show that $H$ maps every one-dimensional stable foliation $\tildeF^s_i (1\leq i\leq k)$  to one of the linearization $\tildeL^s_i$. Moreover, it is an isometry along each leaf of $\tildeF^s_i$. As proved in Proposition \ref{leaf proposition}, we already have that $H$ preserves weak stable foliations.  For reducing to each single leaf, we need the following two propositions.
	The  idea to reduce the leaf conjugacy originated from \cite{Gogolevhighdimrigidity} (also see \cite{GKS2011}). A main tool in \cite{Gogolevhighdimrigidity} is the minimal property of the foliation $\mathcal{F}^s_i$. However, in our case, there is a priori no $\mathcal{F}^s_i (i\geq 2)$ on $\TT^d$. So, we cannot use the minimal property, directly. This obstruction can be overcomed by using a special $\ZZ^d$-sequence described in Proposition \ref{nmH},  Proposition \ref{nmF} and Proposition \ref{projection leaf dense}.
	
	Again, by Proposition \ref{leaf proposition}, we already have $H(\tildeF^s_k)=\tildeL^s_k$. We show that it is an isometry restricted on each leaf of $\tildeF^s_k$ under the metric $d^s_k(\cdot,\cdot)$ given by Proposition \ref{affine metric}. Generally, we have the following proposition.
	\begin{proposition}\label{s-leaf isometric}
		Let $f\in\mathcal{U}$ given by Proposition \ref{leaf proposition}. Fix $1\leq i\leq k$, assume that $H(\tildeF^s_i)=\tildeL^s_i$ and $\lambda^s_i(p,f)=\lambda^s_i(A)$ for all $p\in{\rm Per}(f)$. Then $H$ is isometric restricted on each leaf of $\tildeF^s_i$ under the metric $d^s_i(\cdot,\cdot)$ given by Proposition \ref{affine metric}. Especially,  for an irreducible Anosov map $f$ with dim$E^s=1$, if $\lambda^s(p,f)=\lambda^s_i(A)$ for all $p\in{\rm Per}(f)$, then $H$ is isometric restricted on each leaf of $\tildeF^s$ under the metric $d^s(\cdot,\cdot)$.
	\end{proposition}

	The following proposition allow us to reduce the leaf conjugacy by induction.
	\begin{proposition}\label{s-leaf conjugacy introduction}
			Let $f\in\mathcal{U}$ given by Proposition \ref{leaf proposition}. Fix $1<i\leq k$, assume that $H(\tildeF^s_{(1,i)})=\tildeL^s_{(1,i)}$ and $H$ is isometric restricted on  each leaf of $\tildeF^s_i$ under the metric $d^s_i(\cdot,\cdot)$, then	$H(\tildeF^s_{(1,i-1)})=\tildeL^s_{(1,i-1)}$.
	\end{proposition}
	
		We leave the proofs for Proposition \ref{s-leaf isometric} and Proposition \ref{s-leaf conjugacy introduction}  in subsection \ref{subsection 5.1}. Combining these two propositions, we can prove that $H$ preserves every one-dimensional stable foliation $\tildeF^s_i$ and in fact is an isometry restricted on each leaf of $\tildeF^s_i$.
	\begin{corollary}\label{ key corollary}
		Let $f\in\mathcal{U}$ given by Proposition \ref{leaf proposition} with $\lambda^s_i(p,f)=\lambda^s_i(q,f)$, for all $p,q\in {\rm Per}(f)$ and all $1\leq i\leq k$. Then, for every $1\leq i\leq k$, $H(\tildeF^s_i)=\tildeL^s_i$ and $H$ is isometric along  each leaf of $\tildeF^s_i$ under the metric $d^s_i(\cdot,\cdot)$.
	\end{corollary}

	\begin{proof}[Proof of Corollary \ref{ key corollary}]
	We get the proof by induction.The beginning of the induction is $H(\tildeF^s_{(i,k)})=\tildeL^s_{(i,k)}$ (see Proposition \ref{leaf proposition}), especially,  $H(\tildeF^s_k)=\tildeL^s_k$. Then, by Proposition \ref{s-exp coincide with linear}, the assumption that $\lambda^s_k(p,f)=\lambda^s_k(q,f)$,  for every $p,q\in{\rm Per}(f)$ implies  $\lambda^s_k(p,f)=\lambda^s_k(A)$. It allows us to define an affine metric $d^s_k(\cdot, \cdot)$ by Proposition \ref{affine metric}.
		
		Now, by Proposition \ref{s-leaf isometric}, $H: \tildeF^s_k\to \tildeL^s_k$ is isometric. Thus, by Proposition \ref{s-leaf conjugacy introduction}, $H(\tildeF^s_{(1,k)})=\tildeL^s_{(1,k)}$ implies $H(\tildeF^s_{(1,k-1)})=\tildeL^s_{(1,k-1)}$.
		Moreover, since $H$ preserves the weak stable foliation $H(\tildeF^s_{(k-1,k)})=\tildeL^s_{(k-1,k)}$, we have
		$$H(\tildeF^s_{k-1})=H\left(\tildeF^s_{(1,k-1)}\cap \tildeF^s_{(k-1,k)} \right)\ =\tildeL^s_{(1,k-1)}\cap \tildeL^s_{(k-1,k)}=\tildeL^s_{k-1}.$$
		
		Applying the preceding methods to $\tildeF^s_{k-1}$ and $\tildeF^s_{(1,k-1)}$ , we have $H(\tildeF^s_{(1,k-2)})=\tildeL^s_{(1,k-2)}$. Moreover, by intersecting with $\tildeF^s_{(k-2,k)}$, we have $H(\tildeF^s_{k-2})=\tildeL^s_{k-2}$.
		
		Consequently, we can finish our proof by induction.
	\end{proof}

    Using the isometry $H$ along each leaf of single stable foliations, we also can show that all stable foliations $\tildeF^s_i\; (1\leq i\leq k)$ are $\ZZ^d$-periodic. Moreover, there is no deviation between $H^{-1}(x+n)$ and $H^{-1}(x)+n$ along $\tildeF^s_i$, for all $x\in\RR^d$ and $n\in\ZZ^d$. More precisely, we have the following two propositions.
	
	\begin{proposition}\label{Zd foliation}
		Let $f\in\mathcal{U}$ given by Proposition \ref{leaf proposition}. Assume that 	for every $1\leq j\leq k$, $H(\tildeF^s_j)=\tildeL^s_j$ and $H$ is isometric along  each leaf of $\tildeF^s_j$ under the metric $d^s_j(\cdot,\cdot)$.  Fix $1\leq i<k$, if $\tildeF^s_{(i,k)}$ is $\ZZ^d$-periodic,  then so is $\tildeF^s_{(i+1,k)}$.
	\end{proposition}
	
	\begin{remark}\label{s-dominatedonTd}
	   Using Proposition \ref{Zd foliation} and by induction beginning with the fact that $\tildeF^s_{(1,k)}=\tildeF^s$ is $\ZZ^d$-periodic, we have that $\tildeF^s_{(i,k)} (1\leq i\leq k)$ is $\ZZ^d$-periodic.  Note that $\tildeF^s_{(1,i)} (1\leq i\leq k)$ is always $\ZZ^d$-periodic (see Proposition \ref{leaf proposition}).  Thus, $\tildeF^s_i=\tildeF^s_{(1,i)}\cap \tildeF^s_{(i,k)}$ is also $\ZZ^d$-periodic, for all $1\leq i\leq k$. It follows that for every $1\leq i\leq k$, $\tilde{E}^s_i$ is $\ZZ^d$-periodic. Hence, by \eqref{bundleonorbitandlift}, we get the bundle $E^s_i$ defined well on $\TT^d$.
	\end{remark}

	By Corollary \ref{ key corollary}, $\tilde{E}^s_{(1,i-1)}\oplus \tilde{E}^s_{(i+1,k)}$ is interagble and denote the $F$-invariant integral foliations by $\tildeF^{s,\perp}_i$. We also denote $\tildeL^s_{(1,i-1)}\oplus\tildeL^s_{(i+1,k)}$ by $\tildeL^{s,\perp}_i$.
	
	\begin{proposition}\label{H is no distance on Fi}
			Let $f\in\mathcal{U}$ given by Proposition \ref{leaf proposition}. Assume that 	for every $1\leq i\leq k$, $H(\tildeF^s_i)=\tildeL^s_i$ and $H$ is isometric along  each leaf of $\tildeF^s_i$ under the metric $d^s_i(\cdot,\cdot)$. Then $$H^{-1}(x+n)-n\in \tildeF^{s,\perp}_i\big(H^{-1}(x)\big),$$  for all $1\leq i\leq k$, $x\in\RR^d$ and $n\in \ZZ^d$. Especially, for every irreducible Anosov map $f$ on torus with one dimensional stable bundle, if $H$ is isometric along  each leaf of $\tildeF^s$ under the metric $d^s(\cdot,\cdot)$ then $$H^{-1}(x+n)-n= H^{-1}(x),$$ for all $x\in\RR^d$ and $n\in\ZZ^d$.
	\end{proposition}
	
	\vspace{0.1cm}
	
     We leave the proofs for Proposition \ref{Zd foliation} and Proposition \ref{H is no distance on Fi} in subsection \ref{subsection 5.2}. Now, by the previous propositions, we can prove Theorem \ref{s-rigidity implies special}.
	\begin{proof}[Proof of Theorem \ref{s-rigidity implies special}]
		In the case of dim$E^s=1$, since $\lambda^s(p,f)=\lambda^s(q,f)$ for all $p,q\in{\rm Per}(f)$, combining Proposition \ref{s-leaf isometric} and Proposition \ref{H is no distance on Fi}, we have $H^{-1}(x+n)= H^{-1}(x)+n$, for all $x\in\RR^d$ and $n\in\ZZ^d$. It means that $H$ can descend to $\TT^d$. By Proposition \ref{special and conjugate}, $f$ is special.
		
		In the case of higher-dimensional stable bundle, by Corollary \ref{ key corollary} and Proposition \ref{H is no distance on Fi}, we have
		$$	H^{-1}(x+n)-n \in \bigcap_{i=1}^{k}\left( \tildeF^{s,\perp}_i(H^{-1}x) \right)= \big\{H^{-1}(x)\big\},$$
		for all $x\in\RR^d$ and $n\in\ZZ^d$. Hence, $f$ is special.
		And by Remark \ref{s-dominatedonTd}, $f$ admits $T\TT^d=E^s_1\oplus..\oplus E^s_k\oplus E^u.$
	\end{proof}

	Combining Theorem \ref{special implies s-rigidity} and Corollary \ref{ key corollary}, we can show that, if the conjugacy between the non-invertible Anosov maps $f$ and $A$ exists, then it must be smooth along the stable foliation (see Corollary \ref{cor-1} and Corollary \ref{cor-d}). We mention that it can be proved without Proposition \ref{Zd foliation} and  Proposition \ref{H is no distance on Fi}.
	\begin{proof}[Proof of Corollary \ref{cor-1} and Corollary \ref{cor-d}]
		Assume that $f$ is $C^{1+\alpha}$-smooth. Let $h$ be a conjugacy  between $f$ and $A$. By Proposition \ref{special and conjugate}, $f$ is special. Then, by Theorem \ref{special implies s-rigidity}, we have the spectral rigidity on stable bundle for $f$ which is exactly the condition stated in Theorem \ref{s-rigidity implies special}. Moreover, $f$ admits the finest (on stable bundle) dominated splitting and the conjugacy $h$ maps each stable foliation $\mathcal{F}^s_i$ to $\mathcal{L}^s_i$ (see Proposition \ref{su-integrable and leaf-conjugate}). Since the bundle $E^s_i (1\leq i\leq k)$ is H$\ddot{\rm o}$lder continuous (see Remark \ref{holderbundle}) with exponent $\beta$, for some $0<\beta\leq \alpha$,  by Proposition \ref{Livsic} and the construction of $d^s_i(\cdot,\cdot)$ (see Proposition \ref{affine metric}), the metric $d^s_i(\cdot,\cdot)$ is $C^{1+\beta}$-smooth along each leaf of $\tildeF^s_i$. So, Corollary \ref{ key corollary} actually shows that $h$ is $C^{1+\beta}$-smooth along $\mathcal{F}^s_i$, $1\leq i\leq k$. It follows that $h$ is $C^{1+\beta}$-smooth along the stable foliation $\mathcal{F}^s$, by Journ$\acute{\rm e}$ Lemma \cite{journe}.	In the case of dim$E^s=1$, since each leaf of  the stable foliation $\mathcal{F}^s$ is $C^{1+\alpha}$-smooth, the metric $d^s(\cdot,\cdot)$ is $C^{1+\alpha}$-smooth along each leaf of the  stable foliation $\mathcal{F}^s$ and so is $h$.
	\end{proof}

	\subsection{Induction of the leaf conjugacy} \label{subsection 5.1}
	In this subsection, we prove Proposition \ref{s-leaf isometric} and Proposition \ref{s-leaf conjugacy introduction}. 	Let $f\in\mathcal{U}$ given by Proposition \ref{leaf proposition}. Fix $1\leq i\leq k$, assume that $H(\tildeF^s_i)=\tildeL^s_i$ and $\lambda^s_i(p,f)=\lambda^s_i(A)$ for all $p\in{\rm Per}(f)$. We show that $H$ is isometric restricted on $\tildeF^s_i$.
	
	\begin{proof}[Proof of Proposition \ref{s-leaf isometric}]
		By the existence of affine metric $d^s_i(\cdot,\cdot)$ given by Proposition \ref{affine metric}, we have that $H$ is bi-Lipschitz along $\tildeF^s_i$.
		Indeed, one can just iterate any two points $x_0,y_0\in\RR^d$ such that they are away an almost fixed distance, say $d^s_i\big(F^n(x_0),F^n(y_0)\big)\in [C_1,C_2]$. By the uniform continuity of $H$  and $\tildeF^s_i$ (see Remark \ref{uniform continuity of foliation on Rd}), there exist $C_1', C_2'>0$ such that for every $x,y\in\RR^d$, if $d^s_i(x,y)\in[C_1,C_2]$ then   $d\big(H(x), H(y)\big)\in [C_1',C_2']$. Hence, by Proposition \ref{affine metric} we have
		$$\frac{d\big(H(x_0),H(y_0)\big)}{d_i^s(x_0,y_0)}=\frac{d\big(H\circ F^n(x_0),H\circ F^n(y_0)\big)}{d_i^s\big(F^n(x_0),F^n(y_0)\big)}\in \left[\frac{C_1'}{C_2},  \frac{C_2'}{C_1}  \right].$$
		Similarly, we have $H^{-1}$ is Lipschitz along $\tildeL^s_i$.
		
		It is convenient to prove that $H^{-1}$ is isometric along $\tildeL^s_i$. Thus, so is $H$ along $\tildeF^s_i$.
		
		By Lipschitz continuity, there exists a point $x\in \tildeL^s_i(0)$ differentiable  and we assume $\big(H^{-1}|_{\tildeL^s_i}\big)'(x)=C_0$. It means that, for any $\varepsilon>0$ small enough, there exists $\delta>0$ such that
		\begin{align}
			\left|   \frac{d^s_i\big(H^{-1}(x),H^{-1}(w)\big)}{d(x,w)}-C_0  \right| <\frac{\e}{2}, \label{prop.5.1.0}
		\end{align}
		for every $w\in \tildeL^s_i(x,\delta)$.
	    Fix $y\in\RR^d$, for every $z\in \tildeL^s_i(y)$ and $n_m\in \ZZ^d$, we denote
		$$y_m=\tildeL^s(y+n_m)\cap \tildeL^u(x), \;
		z_m=\tildeL^s_i(y_m)\cap\tildeL^{s,\perp}_i(z+n_m)\;\; {\rm and}\;\;
		x_m=\tildeL^s_i(x)\cap\tildeL^u(z_m) .$$
		
		\begin{claim}\label{claim for H isometric}
		There exists a sequence $\{n_m\}\subset \ZZ^d$ with $n_m\in A^m\ZZ^d$, such that  when $m\to +\infty$,
			\begin{align}
				d(x,x_m)\to d(y,z), \label{proposition5.1-0}
			\end{align}
	    and \begin{align}
				d^s_i\big(H^{-1}(x),H^{-1}(x_m)\big)\to d^s_i\big(H^{-1}(z),H^{-1}(y)\big) .\label{proposition5.1-7}
			\end{align}
		\end{claim}
	
	\begin{proof}[Proof of Claim \ref{claim for H isometric}]
		Fix $m\in\NN$, we consider the point $A^{-m}y$. Since $A$ is irreducible, the unstable foliation $\mathcal{L}^u$ is minimal on $\TT^d$. Thus  we can choose $\tilde{n}_m\in\ZZ^d$ such that $d\big(A^{-m}y+\tilde{n}_m, \tildeL^u(A^{-m}x) \big)\leq 1$.  Let $y(m)= \tildeL^s(A^{-m}y+\tilde{n}_m )\cap  \tildeL^u(A^{-m}x)$ and  $n_m=A^m\tilde{n}_m$.  Note that $y_m=A^my(m)$ and $$d(y_m, y+n_m)\leq \|A|_{L^s}\|^m d(y(m), A^{-m}y+\tilde{n}_m)\to 0,$$ as $m\to +\infty$. The sequence $\{n_m\}$ is what we need.	 Let $m\to+\infty$, we have
		\begin{align}
				d(y_m,y+n_m)\to 0 \quad{\rm and} \quad  d(z+n_m,z_m)\to 0.  \label{prop5.1.1}
			\end{align}
			 Moreover,
			 $$\left| d(y+n_m,z+n_m)-d(y_m,z_m)    \right| \to 0.$$
		    Note that $d(x,x_m)=d(y_m,z_m)$,  then we get \eqref{proposition5.1-0}. See Figure 3.
	
				\begin{figure}[htbp]
				\centering
				\includegraphics[width=14cm]{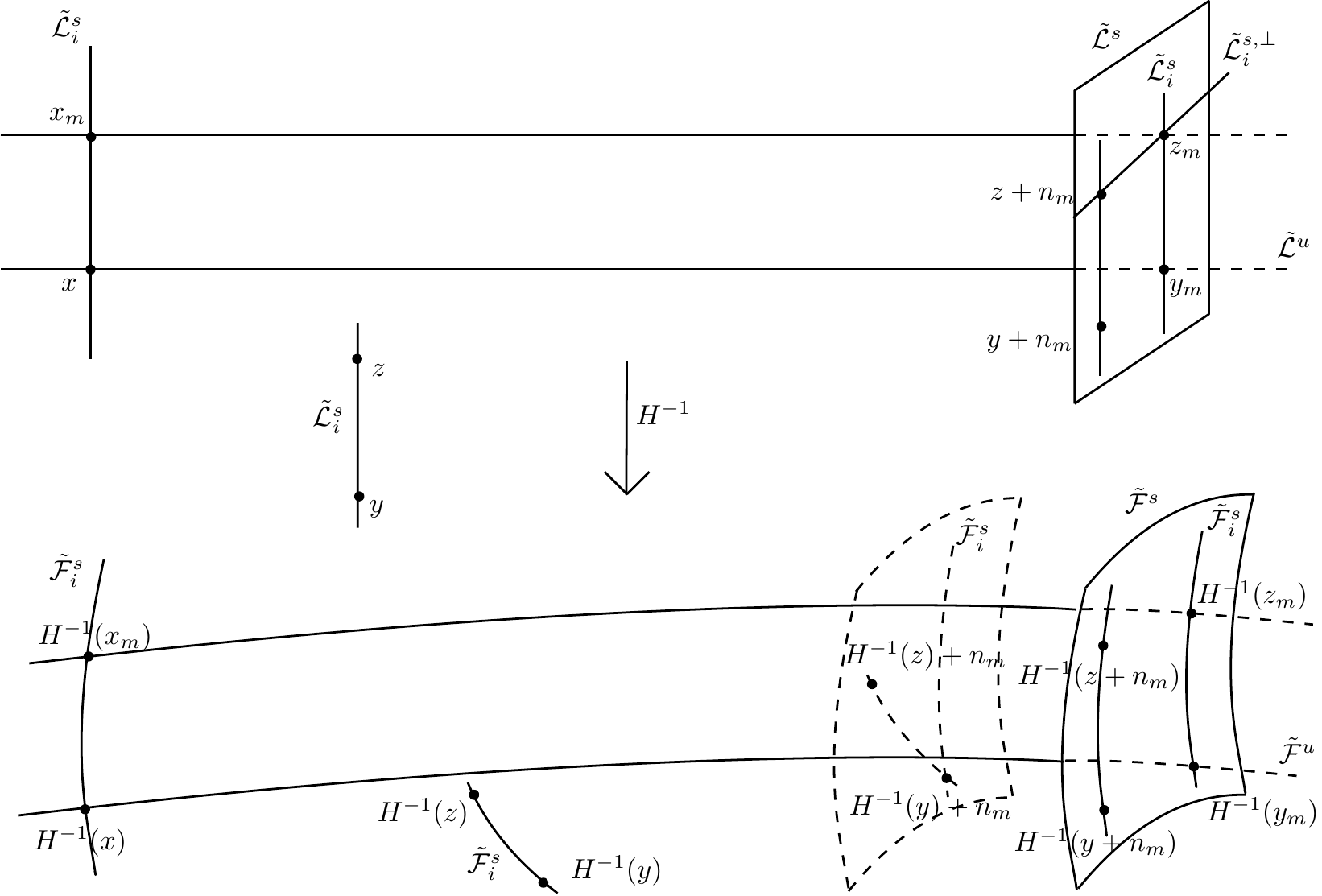}
				\caption{Approaching by $n_m$ sequence.}	
			\end{figure}
			
			Now, by the uniform continuity of $H^{-1}$ and \eqref{prop5.1.1}, one has
			\begin{align}
				d\big(H^{-1}(z_m),H^{-1}(z+n_m)\big)\to 0 \quad {\rm and} \quad d\big(H^{-1}(y_m),H^{-1}(y+n_m)\big)\to 0,  \label{proposition5.1-1}
			\end{align}
			as $m\to +\infty$. By $H^{-1}(\tildeL^s_i)=\tildeF^s_i$ and $H^{-1}(\tildeL^u)=\tildeF^u$, we have $\tilde{E}^s_i\oplus\tilde{E}^u$ is jointly integrable. Thus, by the forth item of Proposition \ref{affine metric}, one has
		\begin{align}
			d^s_i\big(H^{-1}(x),H^{-1}(x_m)\big)=d^s_i\big(H^{-1}(z_m),H^{-1}(y_m)\big).  \label{proposition5.1-2}
		\end{align}
		Combining \eqref{proposition5.1-1} and \eqref{proposition5.1-2},  we get
		\begin{align}
			\Big| d^s_i\big(H^{-1}(x),H^{-1}(x_m)\big)-d^s_i\big(H^{-1}(z+n_m),H^{-1}( y+n_m)\big)   \Big|\to 0. \label{proposition5.1-3}
		\end{align}

	   On the other hand, by Proposition \ref{nmH}, we have
	   \begin{align}
	   	\big(H^{-1}(z+n_m)-n_m\big)\to H^{-1}(z) \quad {\rm and }  \quad \big(H^{-1}(y+n_m)-n_m\big)\to H^{-1}(y), \label{proposition5.1-4}
	   \end{align}
	   as $m\to +\infty$.
		Note that $H^{-1}(z)+n_m$ may not belong to $\tildeF^s_i(H^{-1}(y)+n_m)$.  But, by Proposition \ref{nmF},
		\begin{align}
			d\left( H^{-1}(z)+n_m \;,\;  \tildeF^s_i(H^{-1}(y)+n_m) \right)\to 0. \label{proposition5.1-5}
		\end{align}
		Hence, by \eqref{proposition5.1-4} and \eqref{proposition5.1-5}, one has
		\begin{align}
			d^s_i\big(H^{-1}(z+n_m),H^{-1}(y+n_m)\big)\to d^s_i\big(H^{-1}(z),H^{-1}(y)\big).  \label{proposition5.1-6}
		\end{align}
		Consequently, according to  \eqref{proposition5.1-3} and \eqref{proposition5.1-6}, when $m\to +\infty$, we get \eqref{proposition5.1-7}.
		
	\end{proof}
	
		Now, by Claim \ref{claim for H isometric}, for any $\e>0$ and $z\in \tildeL^s_i(y)$, there exists $N_0\in\NN$ such that when $m>N_0$, one has
		\begin{align}
					\left|   \frac{d^s_i(H^{-1}(y),H^{-1}(z))}{d(y,z)}-\frac{d^s_i(H^{-1}(x),H^{-1}(x_m))}{d(x,x_m)} \right| <\frac{\e}{2}. \label{prop5.1.11}
		\end{align}
		Moreover, let $z\in \tildeL^s_i(y,\frac{\delta}{2})$, we have $d(x,x_m)<\delta$, for all $m>N_0$. Combining \eqref{prop.5.1.0} and \eqref{prop5.1.11},
		 for every $y\in\RR^d$ and every $\varepsilon>0$, there exists $\delta>0$ such that
		$$\left|   \frac{d^s_i\big(H^{-1}(y),H^{-1}(z)\big)}{d(y,z)}-C_0  \right| <\varepsilon,$$
		for every $z\in \tildeL^s_i(y,\frac{\delta}{2})$.   It follows that $H$ is differentiable along $\tildeL^s_i$ and the derivative is the constant $C_0$.  Note that we can change the metric $d^s_i(\cdot,\cdot)$ by scaling such that $C_0=1$.
		
	\end{proof}

	Now, fix $1<i\leq k$, suppose that $H(\tildeF^s_{(1,i)})=\tildeL^s_{(1,i)}$. Note that  since $H$ always preserves weak stable foliations, $H(\tildeF^s_{(1,i)})=\tildeL^s_{(1,i)}$ implies $H(\tildeF^s_i)=H(\tildeF^s_{(1,i)}\cap \tildeF^s_{(i,k)})= \tildeL^s_i$.  Assume that  $H$ is isometric restricted on  each leaf of $\tildeF^s_i$ under the metric $d^s_i(\cdot,\cdot)$. We show $H(\tildeF^s_{(1,i-1)})=\tildeL^s_{(1,i-1)}$.
	\begin{proof}[Proof of Proposition \ref{s-leaf conjugacy introduction}]
		By the assumption $H(\tildeF^s_{(1,i)})=\tildeL^s_{(1,i)}$,
		the foliation $\tildeF=H(\tildeF^s_{(1,i-1)})$ is a sub-foliation of  $\tildeL^s_{(1,i)}$. It is clear that $\tildeF$ and $\tildeL^s_i$ give the Global Product Structure on $\tildeL^s_{(1,i)}$.
		By Proposition \ref{affine metric}, the holonomy maps of $\tildeF^s_i$ along $\tildeF^s_{(1,i-1)}$ restricted on $\tildeF^s_{(1,i)}$ are isometric under the metric $d^s_i(\cdot,\cdot )$.
		Combining with the assumption that $H^{-1}:\tildeL^s_i\to \tildeF^s_i$ is isometric, we have that the holonomy maps of  $\tildeL^s_i$ along $\tildeF$ restricted on $\tildeL^s_{(1,i)}$ are isometric.
		
		Assume that $H^{-1}(\tildeL^s_{(1,i-1)})\neq \tildeF^s_{(1,i-1)}$. It follows that  there exist points $x_0$ and $x_1$ such that
		$x_1\in \tildeF(x_0)\setminus \tildeL^s_{(1,i-1)}(x_0)$.
		\begin{claim}\label{bnyn}
			There exist  $b_m\in \tildeL^s_i(x_0)$, $k_m\in  A^m\ZZ^d$ and $x_2\in \tildeF(x_0)$ such that when $m\to +\infty$,
			\begin{enumerate}
				\item  $(b_m+k_m)\to x_1$.
				\item  $d\left(y_m+k_m, \tildeF(b_m+k_m)\right)\to 0$, where $y_m= \tildeF(b_m)\cap \tildeL^s_i(x_1)$.
				\item $(y_m+k_m)\to x_2$.
			\end{enumerate}
		\end{claim}
		
		\begin{proof}[Proof of 	Claim \ref{bnyn}]
			By the first item of Proposition \ref{projection leaf dense}, we can choose $k_m\in A^m\ZZ^d $   and $b_m\in\tildeL^s_i(x_0)$ such that when $m\to +\infty$, $(b_m+k_m)\to x_1$.
			
			Now, let $y_m= \tildeF(b_m)\cap \tildeL^s_i(x_1)$. Since $|H^{-1}-Id|$ is bounded, $H^{-1}\big(\tildeL^s_i(x_1)\big)$ is located in a neighborhood of $H^{-1}\big(\tildeL^s_i(x_0)\big)$. Hence,  there exists $R>0$ such that  $d_{\tildeF^s_{(1,i-1)}} \big( H^{-1}(b_m), H^{-1}(y_m) \big)<R$, for all $b_m\in\tildeL^s_i(x_0)$. By Proposition \ref{nmH}, $y_m\in\tildeF(b_m)$ implies
			$$ d\left(H^{-1}(y_m) \;,\; \tildeF^s_{(1,i-1)}\left(H^{-1}(b_m+k_m)-k_m,R\right)\right)\to 0.$$
			Since $\tildeF^s_{(1,i-1)}$ is always $\ZZ^d$-periodic (see Proposition \ref{leaf proposition}), one has
			$$d\left(H^{-1}(y_m)+k_m \;,\;  \tildeF^s_{(1,i-1)}\left(H^{-1}(b_m+k_m),R\right)\right)\to 0.$$
			Again, by Proposition \ref{nmH},
			\begin{align}
				d\left(H^{-1}(y_m+k_m) \;,\; \tildeF^s_{(1,i-1)}\left(H^{-1}(b_m+k_m), R \right)\right)\to 0. \label{prop.5.2.1}
			\end{align}
			It follows that
			$$	d(y_m+k_m \;,\;  \tildeF(b_m+k_m))\to0.$$
			So, we complete the second item of claim.

			Note that since $(b_m+k_m)\to x_1$, \eqref{prop.5.2.1} also implies
			$$	d\left(H^{-1}(y_m+k_m) \;,\; \tildeF^s_{(1,i-1)}\left(H^{-1}(x_1), R \right)\right)\to 0.$$
			 Hence, $d\left(y_m+k_m, \tildeF(x_0) \right)\to 0$ and we get the third item.
		\end{proof}

	   Let $z_1=\tildeL^s_i(x_1) \cap \tildeL^s_{(1,i-1)}(x_0)$ and $z_2=\tildeL^s_i(x_2) \cap \tildeL^s_{(1,i-1)}(x_0)$.  Since $x_1\notin\tildeL^s_{(1,i-1)}(x_0)$, we have $d(z_1,x_1)>0$. See Figure 4.
		\begin{claim}\label{claimprop5.2.2}
		 $d(z_2,x_2)=2\cdot d(z_1,x_1)$ and $d(x_0,z_2)=2\cdot d(x_0,z_1)$.
		\end{claim}
		\begin{figure}[htbp]
		\centering
		\includegraphics[width=16cm]{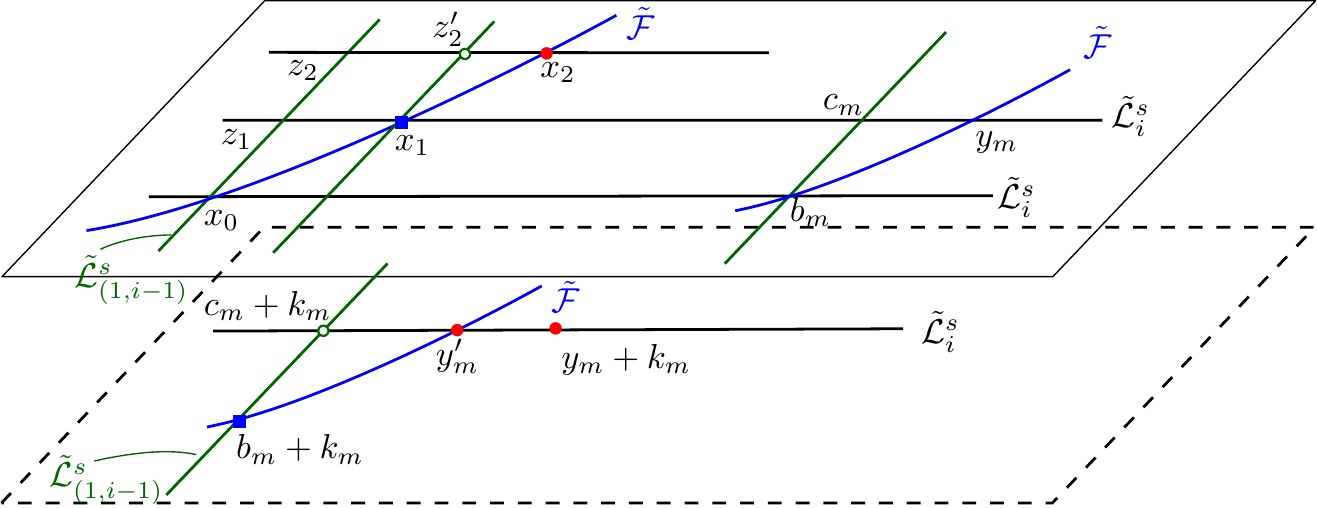}
		\caption{Deduction of $H$ preserving the strong stable foliations.}	
	\end{figure}
	\begin{proof}[Proof of Claim \ref{claimprop5.2.2}]
		Let $y_m\in\ZZ^d$ and $b_m\in \tildeL^s_i(x_0)$ given by Claim \ref{bnyn}. Denote
		$$c_m=\tildeL^s_{(1,i-1)}(b_m)\cap \tildeL^s_i(x_1), \quad z_2'=  \tildeL^s_{(1,i-1)}(x_1)\cap \tildeL^s_i(x_2)  \quad {\rm and } \quad y_m'=\tildeF(b_m+k_m) \cap \tildeL^s_i(y_m+k_m).$$
		By the proof of Claim \ref{bnyn}, we actually have $d(y_m',y_m+k_m)\to 0$, $y_m'\to x_2$ and $c_m+k_m\to z_2'$.
		Since the holonomy maps of  $\tildeL^s_i$ along $\tildeF$ restricted on $\tildeL^s_{(1,i)}$ are isometric, we have $d(c_m, y_m)=d(z_1,x_1)$.  It follows that $d(c_m+k_m, y_m+k_m)=d(z_1,x_1)$.  Hence,
		$$d(z_2',x_2)=\lim_{m\to +\infty}d(c_m+k_m, y_m')=\lim_{m\to +\infty}d(c_m+k_m, y_m+k_m)=d(z_1,x_1).$$
		Since the foliation $\tildeL^s_i$ is one-dimensional, $d(z_2,x_2)=d(z_2,z_2')+d(z_2',x_2)=2d(z_1,x_1)$.
	
		For the other equation,
		 $$d(z_1,z_2)=d(x_1,z_2')=\lim_{m\to+\infty}d(b_m+k_m,c_m+k_m)=\lim_{m\to+\infty}d(b_m,c_m)=d(x_0,z_1).$$
		Note that the dimension of $\tildeL^s_{(1,i-1)}$ could be more than one, but the line $\overline{x_0z_1}$ is parallel to the line  $\overline{b_mc_m}$ and also the line $\overline{(b_m+k_m)(c_m+k_m)}$.  Hence, $\overline{x_0z_1}$  is parallel to $\overline{x_1z_2'}$ and also $\overline{z_1z_2}$. Thus, we have $d(x_0,z_2)=d(x_0,z_1)+d(z_1,z_2)=2d(x_0,z_1)$.
	\end{proof}
	
	Repeating the construction in Claim \ref{bnyn} and  Claim \ref{claimprop5.2.2},
	there exists a sequence $\{ x_l\}\subset \tildeF(x_0), l\in\NN$ with  $z_l= \tildeL^s_i(x_l)\cap \tildeL^s_{(1,i-1)}(x_0)$ such that   $d(z_l,x_l)=l\cdot d(z_1,x_1)$ and $d(x_0,z_l)=l\cdot d(x_0,z_1)$.
	Fix $\delta>0$,	let $N_l>0$ be the minimal number such that $A^{N_l}z_l\in \tildeL^s_{(1,i-1)}(A^{N_l}x_0,\delta)$. Since $\mu^s_i(A)> \mu^s_{i-1}(A)$, we have 	$d(A^{N_l}z_l, A^{N_l}x_l)\to +\infty$. It contradicts with the fact that $d(y,y')$ is bounded, for every $y\in \tildeL^s_{(1,i-1)}(x_0,\delta)$ and $y'=\tildeL^s_i(y)\cap \tildeF(x_0)$.
	\end{proof}

	\subsection{$\ZZ^d$-periodic foliations}\label{subsection 5.2}
	
	Let $f\in\mathcal{U}$ given by Proposition \ref{leaf proposition} with $\lambda^s_i(p,f)=\lambda^s_i(q,f)$, for all $p,q\in {\rm Per}(f)$ and $1\leq i\leq k$. Now, we already have gotten  Corollay \ref{ key corollary} that is for each $1\leq i\leq k$, $H(\tildeF^s_i)=\tildeL^s_i$ and $H$ is isometric along  $\tildeF^s_i$ under the metric $d^s_i(\cdot,\cdot)$.  It follows that the holonomy maps of $\tildeF^s_i$ along  $\tildeF^{s}_{j_1}\oplus \dots\oplus \tildeF^s_{j_l}$ restricted on $\big(\tildeF^{s}_{j_1}\oplus \dots\oplus \tildeF^s_{j_l}\oplus \tildeF^s_i\big) $ are isometric under $d^s_i(\cdot,\cdot)$, where $1\leq j_1<j_2<...<j_l\leq k$.
	
	Fix $1\leq i< k$, assume that $\tildeF^s_{(i,k)}$ is $\ZZ^d$-periodic, we show that $\tildeF^s_{(i+1,k)}$ is also $\ZZ^d$-periodic. Note that the assumption $\ZZ^d$-periodic property for foliation $\tildeF^s_{(i,k)}$  implies one for  $\tildeF^s_i= \tildeF^s_{(1,i)}\cap \tildeF^s_{(i,k)}$.
	
	\begin{proof}[Proof of Proposition \ref{Zd foliation}]
		Assume that there exist $x_0\in\RR^d, n\in\ZZ^d$ and $x_1\in \tildeF^s_{(i+1,k)}(x_0+n)-n$, but $x_1\notin \tildeF^s_{(i+1,k)}(x_0)$.
		Note that, for every $x\in\RR^d$ and $n\in\ZZ^d$,  $\tildeF^s_{(i+1,k)}(x+n)-n \subset \tildeF^s_{(i,k)}(x+n)-n= \tildeF^s_{(i,k)}(x)$.
		Thus we have $x_1\in \tildeF^s_{(i,k)}(x_0)$. Let $z_1=\tildeF^s_i(x_1)\cap \tildeF^s_{(i+1,k)}(x_0)$. By the assumption, $a:=d^s_i(x_1,z_1)>0$.

		For every $b_m\in \tildeF^s_i(x_0)$, we denote $$c_m=\tildeF^s_{(i+1,k)}(b_m)\cap\tildeF^s_i(x_1) \quad {\rm and} \quad y_m=\big(\tildeF^s_{(i+1,k)}(n+b_m)-n\big)\cap\tildeF^s_i(x_1).$$ We claim $d^s_i(y_m,c_m)=a$. Indeed, since $d^s_i(\cdot,\cdot)$ is holonomy invariant, we have
		$$d^s_i(x_0,b_m)=d^s_i(z_1,c_m) \quad {\rm and }\quad  d^s_i(x_0+n,b_m+n)=d^s_i(x_1+n,y_m+n).$$
		Again, $\tildeF^s_{(i,k)}$ is $\ZZ^d$-periodic implies that for $\tildeF^s_i$. It follows that
		$$d^s_i(x_0,b_m)=d^s_i(x_0+n,b_m+n) \quad {\rm and }\quad  d^s_i(x_1+n,y_m+n)=d^s_i(x_1,y_m).$$
		Hence, $d^s_i(y_m,c_m)=d^s_i(x_1,z_1)=a$. See Figure 5.
		
		\begin{figure}[htbp]
			\centering
			\includegraphics[width=16cm]{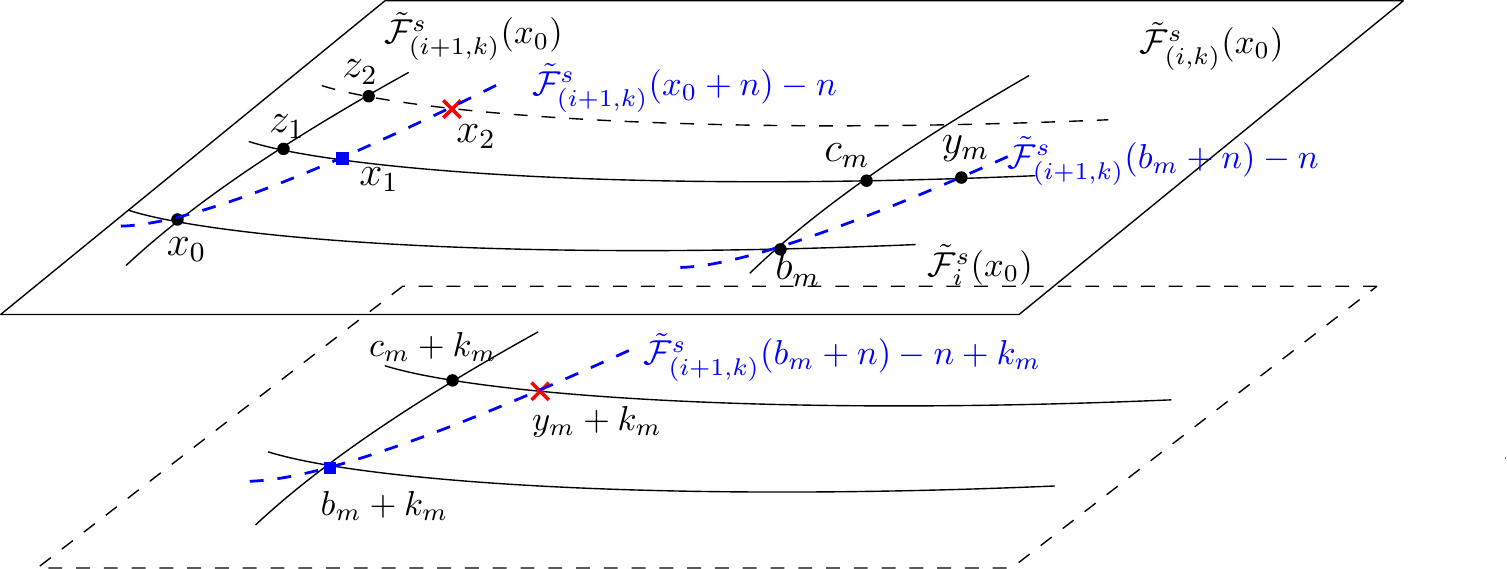}
			\caption{Deduction of $\ZZ^d$-periodic foliations.}	
		\end{figure}

		Since we already have  $H(\tildeF^s_i)=\tildeL^s_i$,  by  the second item of Proposition \ref{projection leaf dense}, we can choose $b_m\in\tildeF^s_i(x_0)$ and $k_m\in A^m\ZZ^d$  such that $(b_m+k_m)\to x_1$.  By Proposition \ref{nmF}, for a fixed size $R>0$, one has
		$$d_H\left(  \tildeF^s_{(i+1,k)}(b_m+n,R)+k_m  \;\;,\;\; \tildeF^s_{(i+1,k)}(b_m+n+k_m,R)    \right) \to 0,$$
		as $m\to +\infty$, where $d_H(\cdot,\cdot) $ is Hausdorff distance.
		Since $(b_m+k_m)\to x_1$ and $(x_1+n)\in  \tildeF^s_{(i+1,k)}(x_0+n)$, we have that
		$$d_H\left(  \tildeF^s_{(i+1,k)}(b_m+n,R)+k_m-n \; \;,\;\;  \tildeF^s_{(i+1,k)}(x_1+n,R) -n   \right) \to 0.$$
		It means that the sequence
		$(y_m+k_m) \in \tildeF^s_{(i+1,k)}(b_m+n)+k_m-n$ converges to $x_2\in \tildeF^s_{(i+1,k)}(x_0+n) -n$.

		Similarly, we can get
		$$d_H\left(  \tildeF^s_{(i+1,k)}(b_m,R)+k_m  \;\;,\;\; \tildeF^s_{(i+1,k)}(x_1,R)    \right) \to 0.$$
		It follows that $(c_m+k_m)\in \tildeF^s_{(i+1,k)}(b_m)+k_m$ converges to $z_2'\in  \tildeF^s_{(i+1,k)}(x_1)$. Moreover, since $\tildeF^s_i$ is $\ZZ^d$-periodic, we have  $$d^s_i(x_2,z_2')=\lim_{m\to+\infty}d^s_i(y_m+k_m,c_m+k_m)=\lim_{m\to+\infty}d^s_i(y_m,c_m)=a.$$
		Let $z_2=\tildeF^s_i(x_2)\cap \tildeF^s_{(i+1,k)}(x_0)$, by the holonomy-invariant metric, we have
		$d^s_i(x_2,z_2)=2a.$

		Repeating the preceding  method, we can pick $x_l \in \tildeF^s_{(i+1,k)}(x_0+n)-n$ and $z_l =\tildeF^s_i(x_l)\cap \tildeF^s_{(i+1,k)}(x_0)$ such that $$d^s_i(x_l,z_l)=la\to +\infty,$$ as $l\to +\infty$.
		Since $\tildeF^s_i$ is quasi-isometric  (Proposition \ref{leaf proposition}), it  follows that
		\begin{align}
		d(x_l,z_l)\to +\infty, \quad {\rm as} \;\;l\to +\infty. \label{la infty}
		\end{align}
	
	On the other hand, since there exists $C_0>0$ such that $|H-Id|<C_0$, one has
		\begin{align*}
			\tildeF^s_{(i+1,k)}(x_0+n)
			& \subset B_{C_0}\left(   \tildeL^s_{(i+1,k)}\big(H(x_0+n)\big)  \right),\\
			& \subset  B_{3C_0}\left(   \tildeL^s_{(i+1,k)}\big(H(x_0)+n\big)  \right),\\
			& = B_{3C_0}\left(   \tildeL^s_{(i+1,k)}\big(H(x_0)\big)  \right)+n\subset
			B_{4C_0}\left(   \tildeF^s_{(i+1,k)}(x_0)  \right)+n.
		\end{align*}
	It follows that  there exists $C\geq 4C_0$ such that  $d(x_l,z_l)\leq C$, for all $l\in\NN$. This contradicts with \eqref{la infty}.
	\end{proof}

	By Remark \ref{s-dominatedonTd}, we have that $\tildeF^s_i (1\leq i\leq k)$ is $\ZZ^d$-periodic. By Corollary \ref{ key corollary}, $H:\tildeF^s_i\to \tildeL^s_i$ is isomtric along each leaf of $\tildeF^s_i$. Now, we can show that there is no deviation between $H^{-1}(x+n)$ and $H^{-1}(x)+n$ along $\tildeF^s_i$. We mention that the following proof can also apply for the case of dim$E^s=1$ without small perturbation, since $H:\tildeF^s\to \tildeL^s$ is also isometric by Proposition \ref{s-leaf isometric}.
	\begin{proof}[Proof of Proposition \ref{H is no distance on Fi}]
		Recall  that $H^{-1}(x+n)-n\in \tildeF^s(H^{-1}(x))$ (see Proposition \ref{H is Zd restricted on stable}). Hence, we just need focus on $\tildeF^s$. For any $x\in\RR^d$ and $y\in\tildeF^s(x)$, let $$\tilde{d}^s_i(x,y):=d^s_i(x, y'),$$
		where $y'= \tildeF^{s,\perp}_i(y)\cap \tildeF^s_i(x)$. Note that the metric $\tilde{d}^s_i(\cdot,\cdot)$ is well defined, since the holonomy maps of $\tildeF^s_i$ along $\tildeF^{s,\perp}_i$ restricted on $\tildeF^s$ are isometric under  $d^s_i(\cdot,\cdot)$. For  proving Proposition \ref{H is no distance on Fi}, it suffices to show that $\tilde{d}^s_i\left(H^{-1}(x) \;,\;  H^{-1}(x+n)-n\right)= 0$, for all $x\in\RR^d$ and $n\in\ZZ^d$.
		
		Assume that there exist $x\in\RR^d$ and $n\in\ZZ^d$ such that $$\tilde{d}^s_i\left(H^{-1}(x) \;,\;  H^{-1}(x+n)-n\right)= \alpha\neq 0.$$
		
		\begin{claim}\label{claim for prop5.6}
		For every $y\in\RR^d$, $\tilde{d}^s_i\left(H^{-1}(y) \;,\; H^{-1}(y+n)-n\right)=\alpha$.
		\end{claim}
	\begin{proof}[Proof of Claim \ref{claim for prop5.6} ]
	By Proposition \ref{projection leaf dense}, there exist $z_m\in \tildeL^s_i(x)$ and $k_m\in A^m\ZZ^d$ such that $(z_m+k_m)\to y$.  By Proposition \ref{nmH} and the uniform continuity of $H$, as $m\to +\infty$,
		\begin{align*}
			d\left(H^{-1}(z_m+k_m) \;,\;  H^{-1}(z_m)+k_m\right)\to 0  \quad &{\rm and } \quad d\left(H^{-1}(z_m+k_m) \;,\; H^{-1}(y)\right)\to 0,\\
			d\left(H^{-1}(z_m+n)+k_m \;,\;  H^{-1}(z_m+n+k_m)\right)\to 0 \quad &{\rm and } \quad  d\left(  H^{-1}(z_m+k_m+n)\;,\; H^{-1}(y+n) \right)\to 0.
		\end{align*}
		Hence,
		\begin{align}
    d\left(H^{-1}(z_m)+k_m \;,\;  H^{-1}(z_m+n)-n+k_m \right)\to d\left(H^{-1}(y) \;,\; H^{-1}(y+n)-n\right). \label{prop5.7.1}
		\end{align}
		On the other hand, since foliations $\tildeF^s_i$ and $\tildeF^{s,\perp}_i$ are both $\ZZ^d$-periodic,
		\begin{align}
			\tilde{d}^s_i\left( H^{-1}(z_m)+k_m \;,\;  H^{-1}(z_m+n)-n+k_m \right)
	        =\tilde{d}^s_i\left( H^{-1}(z_m) \;,\;  H^{-1}(z_m+n)-n\right). \label{prop5.7.2}
		\end{align}
	   Since $H^{-1}$ is isometric along $\tildeL^s_i$, we have
	   \begin{align*}
	   	d^s_i\left( H^{-1}(z_m) \;,\;  H^{-1}(x)\right)
	   	&=d^s_i\left( H^{-1}(z_m+n) \;,\;  H^{-1}(x+n)\right),\\
	   	&=d^s_i\left( H^{-1}(z_m+n)-n \;,\;  H^{-1}(x+n)-n\right).
	   \end{align*}
	   It follows that
	   \begin{align}
	   	\tilde{d}^s_i\left( H^{-1}(z_m) \;,\;  H^{-1}(z_m+n)-n\right)=\tilde{d}^s_i\left( H^{-1}(x) \;,\;  H^{-1}(x+n)-n\right)=\alpha. \label{prop5.7.3}
	   \end{align}
		Now, combining \eqref{prop5.7.1} with \eqref{prop5.7.2} and \eqref{prop5.7.3}, we get $$\tilde{d}^s_i\left(H^{-1}(y) \;,\; H^{-1}(y+n)-n\right)=\lim_{m\to+\infty} \tilde{d}^s_i\left( H^{-1}(z_m)+k_m \;,\;  H^{-1}(z_m+n)-n+k_m \right)=\alpha.$$
			\end{proof}

		For every $l\in\NN$, applying the previous claim for $H^{-1}y=H^{-1}\big(x+(l-1) n\big)-(l-1) n$,  one has
		\begin{align*}
			\tilde{d}^s_i\left(H^{-1}(x+ln)-ln \;,\;  H^{-1}(x)\right)
			&=\tilde{d}^s_i\left(H^{-1}y \;,\; H^{-1}(x)\right)+ \tilde{d}^s_i\left(H^{-1}y \;,\;  H^{-1}(x+ln)-ln \right),\\
			&=(l-1)\cdot \alpha +\alpha =l\cdot \alpha\to +\infty,
		\end{align*}
			as $l\to+\infty$. Since $\tildeF^s_i$ is quasi-isometric (Proposition \ref{leaf proposition}), it contradicts with the fact that for all $x\in\RR^d$ and all $m\in\ZZ^d$, $d\left(H^{-1}(x+m)-m \;,\;  H^{-1}(x)\right)$ is bounded.
		
	\end{proof}

	\flushleft{\bf Jinpeng An} \\
	School of Mathematical Sciences, Peking University, Beijing, 100871,  China\\
	\textit{E-mail:} \texttt{anjinpeng@gmail.com}\\
	
	\flushleft{\bf Shaobo Gan} \\
	School of Mathematical Sciences, Peking University, Beijing, 100871,  China\\
	\textit{E-mail:} \texttt{gansb@pku.edu.cn}\\
	
	\flushleft{\bf Ruihao Gu} \\
	School of Mathematical Sciences, Peking University, Beijing, 100871,  China\\
	\textit{E-mail:} \texttt{rhgu@pku.edu.cn}\\
	
	\flushleft{\bf Yi Shi} \\
	School of Mathematical Sciences, Peking University, Beijing, 100871,  China\\
	\textit{E-mail:} \texttt{shiyi@math.pku.edu.cn}\\

\end{document}